\documentclass[11pt,titlepage]{amsart}
 \usepackage[foot]{amsaddr}

\usepackage{amsmath}
\usepackage{amssymb}
\usepackage{amsthm}
\usepackage{graphicx}
\usepackage{tikz}
\usepackage{xcolor}
\usepackage{hyperref}
\usepackage{comment}
\usepackage{subfigure}

\allowdisplaybreaks

\def\E{\mathbb{E}}
\def\var{\mathbb{Var}}

\newcommand{\bl}{\boldsymbol{\lambda}}

\newcommand{\blambda}{\boldsymbol{\lambda}}

\newcommand{\dist}{\mathrm{dist}}

\newcommand{\Abar}{\overline{A}}

\newcommand{\vu}{\vec{u}}
\newcommand{\ve}{\vec{e}}
\newcommand{\Var}{\operatorname{Var}}

\newcommand{\one}{\mathbf{1}}

\renewcommand{\sp}{\mathrm{sp}}
\newcommand{\SA}{\mathrm{SA}}

\newcommand{\diam}{\mathrm{diam}}

\newcommand{\bz}{\mathbf{z}}

\def\E{\mathbb{E}}
\def\R{\mathbb{R}}
\def\P{\mathbb{P}}

\def\Z{\mathbb{Z}}

\def\eps{\varepsilon}
\def\del{\delta}

\def\bv{\mathbf v}
\def\bx{\mathbf x}

\def\1{\mathbf{1}}

\def\lam {\lambda}
\def\Lam{\Lambda}
\def\tce{t_c + \eps}
\def\tce2{t_c + \frac{\eps}{2}}

\renewcommand{\S}{\mathbb{S}}
\DeclareMathOperator\supp{supp}

\def\bv{\mathbf{v}}

\def\bX{\mathbf{X}}
\def\by{\mathbf{y}}
\newcommand{\TT}{\mathbb{T}}

\def\var{\text{var}}

\newtheorem*{theorem*}{Theorem}
\newtheorem{theorem}{Theorem}
\newtheorem{lemma}[theorem]{Lemma}
\newtheorem{cor}[theorem]{Corollary}

\newtheorem{prop}[theorem]{Proposition}
\newtheorem*{prop*}{Proposition}

\newtheorem*{conj*}{Conjecture}
\newtheorem{claim}[theorem]{Claim}
\newtheorem{question}{Question}

\newtheorem*{fact*}{Fact}
\newtheorem{fact}[theorem]{Fact}

\theoremstyle{definition}
\newtheorem{defn}[theorem]{Definition}
\newtheorem*{defn*}{Definition}

\theoremstyle{remark}

\newtheorem*{remark*}{Remark}
\newtheorem{example}{Example}
\newcommand{\cP}{\mathcal{P}}
\newcommand{\cF}{\mathcal{F}}
\newcommand{\vn}{\vec{n}}

\begin{document}
	\title[Strong spatial mixing  for repulsive point processes]{Strong spatial mixing  for repulsive \\point processes}
	\author{Marcus Michelen}
	\address{Department of Mathematics, Statistics, and Computer Science\\ University of Illinois at Chicago}
	\author{Will Perkins}
	\address{School of  Computer Science\\ Georgia Institute of Technology}
	\email{michelen.math@gmail.com, \, \\  math@willperkins.org}
	\date{\today }

	\begin{abstract}
		We prove that a Gibbs point process interacting via a finite-range, repulsive potential $\phi$ exhibits a strong spatial mixing property for activities $\lam < e/\Delta_{\phi}$, where $\Delta_{\phi}$ is the potential-weighted connective constant of $\phi$, defined recently in~\cite{mp-CC}.  Using this we derive several analytic and algorithmic consequences when $\lam$ satisfies this bound:
\begin{enumerate}
\item We prove  new identities for the infinite volume pressure and surface pressure of such a process (and in the case of the surface pressure establish its existence).
\item We prove  that local block dynamics for sampling from the model on a box of volume $N$ in $\R^d$ mixes in time $O(N \log N)$, giving efficient randomized algorithms to approximate the partition function  and approximately sample from these models.
\item We use the above identities and algorithms to give efficient approximation algorithms for the pressure and surface pressure.
\end{enumerate}
	\end{abstract}

	\maketitle

\section{Introduction}

A Gibbs point process (or a classical gas) is a model of a gas or a fluid with particles interacting in the continuum via a pair (or multibody) potential (see Ruelle's classic reference on the topic for background~\cite{ruelle1999statistical}).  To understand the equation of state of a gas modeled by a Gibbs point process, one needs to compute the infinite volume pressure and density and understand their relationship.  Finite-volume effects can be understood by computing the surface pressure of the model.   The classic rigorous approach to this problem is to use convergent series expansions, like the Mayer and virial series~\cite{mayer1941molecular,penrose1963convergence,ruelle1963correlation,lebowitz1964convergence}, while sampling  methods like  Markov chain Monte Carlo  and molecular dynamics are used experimentally to understand these models~\cite{metropolis1953equation,alder1957phase,Krauth}.   There are theoretical and practical limitations to all of these methods, related to both analytic and computational obstacles.  For instance,  expansion methods are limited to the domain of convergence of the relevant series, and the Markov chain Monte Carlo approach is guaranteed to be efficient only in parameter regimes in which there is rapid convergence to the equilibrium distribution.

 Gibbs point processes are also used to model a wide variety of phenomena in other fields, from the growth of trees in a forest, to the locations of galaxies in the universe, to the spatial and temporal occurrences of earthquakes.  See~\cite{moller2007modern,daley2007introduction} for an introduction from this perspective.  
 To do modeling work with a  spatial point process one would like to sample from the process or compute its associated statistics.  Whether these tasks can be accomplished efficiently also depends on the model and parameters.  There are a wide variety of approaches to sampling from point processes including approximate sampling algorithms based on Markov chains~\cite{preston1975spatial,moller1989rate} and perfect sampling algorithms based on the technique of coupling from the past~\cite{propp1996exact,haggstrom1999characterization,garcia2000perfect,moller2001review,huber2016perfect}.

We use recently developed recursive identities for the density of a Gibbs point process to establish that  Gibbs point processes interacting via repulsive pair potentials exhibit  \textit{strong spatial mixing}  in a wider range of parameters than previously known.  This allows us to prove new identities for the pressure and surface pressure of the models and to provide new algorithms for the above problems with rigorous guarantees of efficiency.   Our approach is inspired by techniques from computer science, including Weitz's approach to approximate counting~\cite{Weitz}, the refinement of this method based on connective constants~\cite{sinclair2017spatial}, and connections between computational and probabilistic properties of discrete statistical mechanics models~\cite{dyer2004mixing,gamarnik2009sequential}.

We begin by defining the point processes we study.
 
 A Poisson process of intensity or activity $\lam$ is the benchmark spatial point process with non-interacting points: the number of points appearing in a region $B$ has a Poisson distribution with mean $\lam |B|$ where we write $|B|$ for the Lebesgue measure of a set, and the numbers of points appearing in disjoint regions are independent.  
  
A Gibbs point process interacting via a repulsive pair potential is defined by a density with respect to  a Poisson process.  For a pairwise interacting Gibbs point process this density is given by $e^{-H(\cdot)}$ where
\begin{align*}
H(x_1, \dots , x_k) = \sum_{1 \le i < j \le k} \phi (x_i-x_j)   \,,
\end{align*}
where $\phi : \R^d \to (-\infty, + \infty]$ defines the (translation invariant) \textit{pair potential} which satisfies $\phi(x) = \phi(-x)$.  The potential is \emph{repulsive} if $\phi(x) \ge 0$ for all $x$.   The potential is of \emph{finite range} if there exists $r \ge 0$ so that $\phi(x) = 0$ when $\| x \|>r$, where $\| \cdot \|$ is the Euclidean norm on $\R^d$. Our results will apply to finite-range, repulsive pair potentials.  Examples of point processes defined by finite-range repulsive pair potentials include the Strauss process~\cite{strauss1975model,kelly1976note}, a model of anti-clustering that penalizes pairs of nearby points, and the hard sphere model~\cite{alder1957phase,lowen2000fun}, a long-studied model of a gas in statistical physics.   See Section~\ref{secExamples} below for more on these examples.

Now fix a pair potential $\phi$.  For a bounded, measurable region $\Lam \subset \R^d$ and activity $\lam \ge 0$, the \emph{partition function} of the Gibbs point process is
\begin{equation}
\label{eqPPpartition1}
Z_{\Lam}(\lam) = 1+ \sum_{k \ge1 } \frac{\lam^k}{k!}   \int_{\Lam^k}  e^{-H(x_1, \dots ,x_k)}  \, dx_1 \cdots dx_k \,.
\end{equation}
The probability measure $\mu_{\Lam,\lam}$ on finite point sets in $\Lam$ is defined by
\begin{equation} \label{eq:mu-def1}
\mu_{\Lam,\lam}(A) =\frac{1}{Z_{\Lam}(\lam)} \sum_{k \geq 0} \frac{\lam^k}{k!}\int_{\Lam^k}\one_{\{x_1, \dots, x_k \} \in A} e^{-H(x_1,\ldots,x_k)} \,dx_1\,\cdots\,dx_k \,,
\end{equation}
where $A \subseteq \bigcup_{k \ge 0} \Lam^k$. 
By taking appropriate limits, one can obtain the infinite volume pressure and surface pressure of the model.  Let $\Lambda_n = [-n,n]^d$ be the axis parallel box of side length $2n$ centered at the origin in $\R^d$, and let $|\Lambda_n| = (2n)^d$ denote its volume.  The \emph{infinite volume pressure} (or simply, the pressure) is
\begin{equation}
\label{eqPressureDef}
p(\lam) = \lim_{n \to \infty} \frac{1}{|\Lambda_n|} \log Z_{\Lambda_n} (\lam) \,.
\end{equation}
Under very general conditions this limit exists (see e.g.~\cite{ruelle1999statistical}), and the non-analytic points of $p(\lam)$ on the positive real axis mark phase transitions of the model~\cite{yang1952statistical}.  Additionally, we may replace the sequence of boxes $\Lambda_n$ with other growing shapes such as balls and obtain the same limit.

The \emph{surface pressure} along a box is defined by
 \begin{equation}
 \label{eqSPressureDef}
 \sp_{\Lambda}(\lambda) = \lim_{n \to \infty} \frac{1}{ \SA(\Lam_n) } \left(\log Z_{\Lambda_n}(\lambda) - |\Lambda_n| p(\lambda) \right)\,,
 \end{equation}
 where $\SA(\Lam_n) = d2^{d}n^{d-1}$ is the surface area of $\Lam_n$. 
The surface pressure is a measure of the first-order correction to the pressure in finite volume. Unlike the pressure, it is not immediately clear that the limit in~\eqref{eqSPressureDef} exists; further, for some potentials this limit may depend on the shape of the finite-volume regions  considered, e.g. the limit may be different if the box $\Lambda_n$ is replaced with a ball.  As such, when referring to surface pressure, we use a subscript to denote the shape along which the limit is taken, with $\Lam$ representing a box and $B$ representing a ball.

Because of the presence of the partition function in defining the probability measure, a Gibbs point process is not analytically tractable in general, hence the need for alternative computational methods such as Markov chain Monte Carlo or series expansion methods.  

There are two central and related computational problems associated to these  processes. The \textit{approximate sampling problem} asks for a randomized algorithm to output a point set in $\Lam$ with distribution within $\eps$ total variation distance of $\mu_{\Lam,\lam}$.    The approximate counting problem asks for an algorithm to output an $\eps$-relative approximation to $Z_{\Lam}(\lam)$ (we give more formal definitions in Section~\ref{secBlockDynamics}).   We say such algorithms are efficient if they run in time polynomial in the volume of $\Lam$ and in $1/\eps$.

In general, we may hope that when the interaction strength is weak enough or the activity $\lam$ is small enough, efficient computational methods exist, just as it is known that no phase transition  occurs with weak enough interaction and small enough activity.  One measurement of the strength of the potential $\phi$ is its temperedness constant  
 $$C_\phi = \int_{\R^d}|1 - e^{-\phi(w)}|\,dw \,.$$
 For the hard sphere model, $C_{\phi}$ is simply the volume of the ball of radius $r$.  In~\cite{michelen2020analyticity} it was shown using computational recursions inspired by~\cite{Weitz}  that there is no phase transition for $\lam < e /C_{\phi}$ for any repulsive potential $\phi$.   Then in~\cite{mp-CC},  a \emph{potential-weighted connective constant} $\Delta_{\phi}$ was defined that captures the interplay between the strength of the potential and the geometry of the underlying space.  We now recall the definition of $\Delta_\phi$.  Let
{\small
\begin{equation} \label{eq:Vk-def}
V_k = \int_{(\R^d)^k} \prod_{j = 1}^k \left(\exp\left(-\sum_{i = 0}^{j-2} \one_{\|v_j-v_i\| < \|v_i-v_{i+1}\|} \phi(v_j -v_i) \right)\cdot \left(1 - e^{-\phi(v_j -v_{j-1})}\right) \right)\,d\bv
\end{equation} }
where we write $d\bv = dv_1\,dv_2\,\ldots\,d_k$, interpret $v_0 = 0$ and in the case of $j = 1$ interpret the empty sum as equal to $0$.  The sequence $V_k$ is submultiplicative, and so we may define the \emph{potential-weighted connective constant} $\Delta_\phi$ via 
\begin{equation}\label{eq:CC-def}
\Delta_\phi := \lim_{k \to \infty} V_k^{1/k} = \inf_{k \geq 1} V_k^{1/k}\,.
\end{equation}
Immediately from the definition we have that $\Delta_\phi \leq C_\phi$, and the inequality is strict for any non-trivial potential on $\R^d$.   For instance, for hard spheres in dimension $2$, $\Delta_{\phi} \le .841 \cdot C_{\phi}$.
In~\cite{mp-CC}, uniqueness of  the infinite volume Gibbs measure and analyticity of the pressure was proved for $\lam < e /\Delta_{\phi}$ by understanding the interplay between the computational recursions of~\cite{michelen2020analyticity} and the geometry of the underlying space captured by $\Delta_{\phi}$. 

Here we prove that a strong correlation decay property (\textit{strong spatial mixing}, Definition~\ref{defSSM} below) holds for finite-range potentials under the  condition $\lam < e /\Delta_{\phi}$,  and we use this to prove both algorithmic and probabilistic results.  For this range of parameters (which goes well beyond the regime of cluster expansion convergence) we:
\begin{enumerate}

\item Prove new identities for the pressure and surface pressure (Theorems~\ref{lem:pressure-dir} and~\ref{lem:surface-pressure}).  In particular, this establishes the existence of the limit defining the surface pressure in~\eqref{eqSPressureDef} for a wider range of parameters than previously known.  
\item Establish a bound on the finite-volume correction to the pressure in a box with periodic boundary conditions that is exponentially small in the sidelength (Theorem~\ref{lem:sp-torus}).  
\item  Provide efficient randomized approximate counting and sampling algorithms for the finite-volume Gibbs measure and partition function respectively (Theorem~\ref{thmFastMixing} and Corollary~\ref{thmApproxAlg}).
\item Use these above identities and the algorithmic results to provide efficient approximation algorithms for the pressure and surface pressure (Theorem~\ref{thm:pressure-algs}).  
\end{enumerate}

To state our results precisely we first provide a very general definition of strong spatial mixing for repulsive Gibbs point processes.  While spatial mixing properties typically are defined in terms of the effect of different boundary conditions on a probability measure, we  generalize this by considering the effect of modifying the activity of a point process at different locations, which leads to the study of point processes with non-uniform activity functions.		
	
\subsection{Strong spatial mixing}
\label{subsecSSM}

The notion of \textit{strong spatial mixing} from the study of lattice and other discrete spin systems (e.g.~\cite{dobrushin1985completely,stroock1992equivalence,stroock1992logarithmic,martinelli1999lectures,martinelli1994approach,dyer2004mixing,sinclair2017spatial}) captures in a robust way the degree to which boundary conditions can affect the distribution of a Gibbs measure.  For short-range spin models on $\mathbb Z^d$, the importance of strong spatial mixing is that it is equivalent to the fast mixing of Glauber dynamics (e.g.,~\cite{stroock1992equivalence}).   In computer science, this connection between strong spatial mixing and algorithms plays an important  role since it is the criteria for the success of a family of approximate counting and sampling algorithms using the `method of correlation decay' due to Weitz~\cite{Weitz}. 

Here we define a notion of strong spatial mixing for Gibbs point processes interacting via  repulsive potentials.  The definition is very general and makes use of an inhomogeneous generalization of the reference Poisson process, in which the activity $\lam$ is replaced by an activity function $\bl$.   Our proofs will also make essential use of this generalization. 
  
  For an integrable function $\bl : \R^d \to \R_{\ge 0}$ with bounded support and a pair potential $\phi$,  the partition function of the Gibbs point process with potential $\phi$ and activity function $\bl$ is
\begin{equation}
\label{eqPPpartition}
Z(\bl) = 1+ \sum_{k \ge1 } \frac{1}{k!}   \int_{(\R^d)^k}  \bl(x_1) \cdots \bl(x_k) e^{-H(x_1, \dots ,x_k)}  \, dx_1 \cdots dx_k \,.
\end{equation}	
The corresponding Gibbs measure  $\mu_{\bl}$ is defined by
\begin{equation} \label{eq:mu-def}
\mu_{\bl}(A) = \sum_{k \geq 0} \frac{1}{k!}\int_{(\R^d)^k}\one_{\{x_1, \dots, x_k \} \in A} \frac{\bl(x_1)\cdots \bl(x_k)}{Z(\bl)} e^{-H(x_1,\ldots,x_k)} \,dx_1\,\cdots\,dx_k \,.
\end{equation}
  The model defined above in~\eqref{eq:mu-def1} can be recovered by taking $\bl = \lam$ on $\Lam$ and $0$ elsewhere.  

The connection to boundary conditions is as follows.  Suppose we consider the Gibbs point process on $\Lam$ with activity $\lam$ but fix points at $y_1, \dots, y_t \in \Lam$.  The distribution of the other points is now affected by the presence of these points via the pair potential $\phi$.  The effect can be accounted for by changing the activity of each point in $\Lam$ as follows:
\[ \bl(x) = \lam e^{ - \sum_{j=1}^t \phi(x-y_j)} \,,\]
and so this new activity function $\bl$ implements the boundary conditions imposed by the points  $y_1, \dots, y_t$.  On the other hand, there are many activity functions $\bl$ that cannot be implemented by boundary conditions, and so it is a genuine generalization.  It will be important in our proofs that for a repulsive potential modifying any activity function in this way  can only decrease it pointwise.    

Strong spatial mixing   measures   how close the distributions of $\mu_{\bl}$ and $\mu_{\bl'}$ are, projected to a set $\Lam$, when $\bl$ and $\bl'$ differ only on a set $A$ far from $\Lam$.   To make this precise we need several definitions.
	
	For two probability measures $\mu, \mu'$ on the same sample space equipped with the same $\sigma$-algebra the \textit{total variation distance} between the two measures is 
	\[ \| \mu - \mu ' \|_{TV} =  \sup_A | \mu(A) - \mu'(A)|  \, ,\]
	where the supremum is over all events $A$.   
	For an activity function $\bl$ and  region $ \Lam \subset \R^d$, let $\mu_{\bl}^\Lam$ denote the law of the point process $\mu_{\bl}$ projected to $\Lam$; one may think of the projected measure $\mu_{\bl}^\Lam$ as being given by samples $\mu_{\bl}$ and then looking only at the point set within $\Lam$. 
	Then for  activity functions $\bl, \bl'$ we define 
	\begin{align*}
	\| \mu_{\bl} - \mu_{\bl'} \|_{\Lam} = \| \mu_{\bl}^\Lam - \mu_{\bl'}^\Lam \|_{TV} \,.
	\end{align*}

	We now formally define the notion of projecting the law of a point process $\mu_{\bl}$ to a region $\Lam \subset \R^d$.   Let $\mathcal N_f$ denote the finite point sets in $\R^d$ and $\mathfrak N(\Lam)$ be the $\sigma$-field generated by finite point sets in $\Lambda$.  For a set $A \in \mathfrak N(\Lam)$ we can define the probability of $A$ under the measure $\mu_{\bl}$ projected to $\Lam$ by defining  $ A_{\Lam} = \{ \eta \in \mathcal N_f : \eta \cap \Lam \in A \}$;  that is, $ A_{\Lam}$ is the set of all finite counting measures on $\R^d$ whose restriction to $\Lam$ is in $A$.  We then define $\mu_{\bl}^\Lam(A)  = \mu_{\bl}( A_{\Lam})$.   We thus obtain the identity
	\begin{equation}
	\label{eqTVProj}
	\| \mu_{\bl} - \mu_{\bl'} \|_{\Lam} = \sup_{A \in \mathfrak N(\Lam) } | \mu_{\bl}( A_{\Lam}) - \mu_{\bl'}( A_{\Lam}) | \,.
	\end{equation}
	
	We denote the support of a function $f$ by $\supp(f) = \{x \in \R^d : f(x) \ne 0 \}$.  For a measurable set $S \subset \R^d$ we let $|S|$ denote its Lebesgue measure.  For two sets $A$ and $B$ in $\R^d$ set $\dist(A,B) = \inf_{x \in A,y \in B} \|x - y\|$; for the case of $A = \{v\}$ we will also write $\dist(v,B) = \dist(\{v\}, B)$.  We can now define strong spatial mixing.
	\begin{defn}
		\label{defSSM}
		The family of point processes on $\R^d$ defined by a repulsive pair potential $\phi$ exhibits \textbf{strong spatial mixing} with activities bounded by $\lam>0$ if there exist constants $\alpha, \beta >0$ so that the following holds: for any bounded, measurable region $\Lam \subset \R^d$ and any two activity functions $\bl, \bl'$ bounded by $\lam$:
\begin{equation}
\label{eqSSMdef}
 \| \mu_{\bl} - \mu_{\bl'} \|_{\Lam} \le \alpha | \Lam|  e^{-\beta \cdot \mathrm{dist}(\Lam, \supp(\bl-\bl')) } \, .
 \end{equation}
	\end{defn}

Our first main result is that strong spatial mixing holds for finite-range repulsive potentials for activities less than $e/ \Delta_{\phi}$. 
	\begin{theorem}
		\label{thmSSM}
		Let $\phi$ be a finite-range, repulsive potential and let $\Delta_{\phi}$ be the potential-weighted connective constant defined in~\eqref{eq:CC-def}.
		Then for any $\lam \in [0, e/ \Delta_{\phi})$, the family of point processes defined by $\phi$  exhibits strong spatial mixing with activities bounded by $\lam$.  
	\end{theorem}

 Strong spatial mixing for a finite-range potential follows from a convergent cluster expansion (e.g.~\cite{ueltschi2004cluster}); this is known to hold (in the repulsive case) for activities bounded by $1/(e C_{\phi})$~\cite{groeneveld1962two} with some recent improvements for specific potentials~\cite{fernandez2007analyticity,jansen2019cluster,nguyen2020convergence}.  It is also known that the cluster expansion \textit{cannot} converge for activities larger than  $1/C_{\phi}$, and so Theorem~\ref{thmSSM} improves the known bound by a factor at least  $e^2$ and the known limit of the cluster expansion approach by a factor $e$.   In the specific case of the hard sphere model, Theorem~\ref{thmSSM} improves the bound for strong spatial mixing from~\cite{helmuth2020correlation} by a factor at least $e/2$.   
 
 The method of disagreement percolation~\cite{van1993uniqueness,van1994percolation} can be used in the discrete setting (Ising and hard-core models) to prove uniqueness of Gibbs measure and strong spatial mixing.  Recently disagreement percolation has been applied to Gibbs point processes to prove uniqueness results~\cite{hofer2019disagreement,benevs2019decorrelation,last2021disagreement,betsch2021uniqueness}, and it is quite possible that this method can also prove strong spatial mixing in this setting.  In high dimensions (for, say, the hard sphere potential) the bound for strong spatial mixing obtained in this manner will necessarily be worse than the bound of Theorem~\ref{thmSSM} by a factor tending to  $e$ as $d \to \infty$; however in  dimension $2$,  high-confidence simulations suggest that an improvement may be possible with this method (see the discussion in~\cite[Section 5]{betsch2021uniqueness}).

\subsection{Results  for the pressure and surface pressure}
\label{subsecPressure}
	We now give new identities for thermodynamic quantities such as the pressure and the surface pressure under the assumption of strong spatial mixing, inspired by the \textit{sequential cavity method} of Gamarnik and Katz~\cite{gamarnik2009sequential}.   
	
	These identities will be in terms of {one-point densities} of point processes.  The one-point density $\rho_{\bl}(\cdot)$ is  the density, with respect to Lebesgue measure,  of the measure computing the expected number of points of the point process $\mu_{\bl}$ in a given set.  It can be written in terms of an expectation.   For a point $v \in \R^d$ and an instance of the point process $\mathbf{X}$, define the \emph{added energy} $H_v(\bX)$ to be the random variable $H_v(\bX) := H(v,\mathbf{X}) - H(\mathbf{X})$.  Then the \textit{one-point density} of $\mu_{\bl}$ at $v$ is
	\begin{equation}\label{eq:added-energy}
	\rho_{\bl}(v) = \bl(v) \E_{\mu_{\bl}} e^{-H_v(\bX)}
	\end{equation}
	where the expectation is over the Gibbs point process with activity function $\bl$.  
	
	We will also want to define one-point densities for activity functions $\bl$ that do not have bounded support; for instance, for the constant function $\bl \equiv \lam$. In general this density may not be well defined, since there may be multiple infinite volume Gibbs measures on $\R^d$ with activity $\lam$.  However, under the assumption of strong spatial mixing the density is well defined and we can express it as a limit:
\begin{equation}
\label{eqInfVolDensity}
\rho_{\bl} (v) = \lim_{n \to \infty} \rho_{\bl_n}(v) 
\end{equation}
where $\bl_n(y) = \mathbf 1_{y \in B_n} \cdot \bl(y)$ and $B_n$ is the ball of radius $n$ centered at the origin.    Since $\phi$ is of finite range, the added energy is local, and so the limit of expectations defining $\rho_{\bl} (v) $ exists by the definition of strong spatial mixing.

	Now recall the definitions of the pressure and surface pressure from~\eqref{eqPressureDef} and~\eqref{eqSPressureDef}.  For a unit vector $\vu \in \S^{d-1}$ and $\lambda > 0$, define $\bl_{\vu}$ by $$\bl_{\vu}(x) = \lam\one\{\langle \vu , x \rangle \geq 0 \} \,.$$
	    The function $\bl_{\vu}$ does not have bounded support, so define $\rho_{\bl_{\vu}}$ via~\eqref{eqInfVolDensity}.  Our next main result is that the infinite volume pressure is given by the evaluation of a single one-point density.

	\begin{theorem}\label{lem:pressure-dir}
	Let $\lambda > 0$ so that we have strong spatial mixing for all $\bl \leq \lambda$.  Then for any $\vu \in \S^{d-1}$ we have 
	\begin{equation}\label{eq:pressure-identity}
	p(\lambda) = \rho_{\bl_{\vu}}(0)\,.
	\end{equation}
\end{theorem}

 We next  show that under the assumption of strong spatial mixing, the limit defining the surface pressure in~\eqref{eqSPressureDef} exists,  and we prove a useful identity for it.  
\begin{theorem}\label{lem:surface-pressure}
	Let $\lambda > 0$ so that we have strong spatial mixing for all $\bl \leq \lambda$.  Then  
	\begin{equation}\label{eq:surf-pressure-identity}
	\sp_{\Lambda}(\lambda) = \frac{1}{d} \sum_{j = 1}^d \int_{0}^\infty \int_0^1 \frac{\rho_{t\bl_{\ve_j}}(s\ve_j) - \rho_{t\lambda}(0)}{t} \,dt\,ds\,.
	\end{equation}
\end{theorem}
Here and throughout we write $\{\vec{e}_j\}_{j = 1}^n$ for the standard basis of $\R^n$.  By applying a rotation, we note that Theorem \ref{lem:surface-pressure} holds for any choice of orthonormal basis.  
Note also that by \eqref{eq:added-energy} the integrand of \eqref{eq:surf-pressure-identity} is uniformly bounded by $\lambda$; under the assumption of strong spatial mixing, the integrand decays exponentially in $s$ uniformly in $t$, and so this integral converges.
In Section~\ref{secSurface} we derive further identities for the surface pressure and compute  the limiting surface pressure when taking limits along other shapes, such as spheres and dilations of convex polytopes.  In particular, Theorem \ref{lem:surface-pressure} is a special case of Proposition \ref{pr:polytope}, which computes the surface pressure along a convex polytope.    In the case when $\phi$ is spherically symmetric,  these limits are the same and independent of the polytope.

Finite-volume corrections to the pressure were  studied by Fisher and Lebowitz~\cite{fisher1970asymptotic} who showed that taking the limit in~\eqref{eqPressureDef} along a sequence of boxes with periodic boundary conditions (a sequence of torii) yields the same limiting pressure.  
Fisher and Caginalp~\cite{fisher1977wall} proved the existence of the surface pressure for ferromagnetic lattice systems.  The questions of surface pressure and finite-volume corrections have subsequently been studied in great detail in lattice systems, usually using either specific properties of a model like the dimer or Ising models, or using a convergent cluster expansion (e.g.,~\cite{fisher1961statistical,ferdinand1967statistical,fisher1967interfacial,caginalp1979wall,bricmont1980surface,frohlich1987semi,borgs1990rigorous,borgs1991finite}).

For classical gasses in the continuum, much less is known.  Existence of the surface pressure and an infinite series for its value follows from a convergent cluster expansion (e.g.~\cite{ueltschi2004cluster}).  Pulvirenti and Tsagkarogiannis~\cite{pulvirenti2015finite} prove an upper bound proportional to surface area on finite-volume corrections to the pressure in the canonical ensemble using the respective cluster expansion.   To the best of our knowledge, Theorem~\ref{lem:surface-pressure} is the first to establish the existence of the surface pressure for classical gasses beyond the regime of cluster expansion convergence.

Crucially, the identities in Theorems \ref{lem:pressure-dir} and \ref{lem:surface-pressure}  are given in terms of one-point densities which can be estimated via random sampling.  This allows us to obtain randomized algorithms to approximate the pressure and surface pressure in Theorem~\ref{thm:pressure-algs} below.

Finally, we see that strong spatial mixing implies that on the $d$-dimensional (flat) torus the finite volume pressure converges to the infinite volume pressure exponentially quickly.  We note that this also recovers the  fact that the infinite volume pressure  on the torus is equal to that of euclidean space, which was proven in~\cite{fisher1970asymptotic} without the assumption of strong spatial mixing.  Let $\TT_n^d :=  \R^d / (n \Z)^d$ denote the $d$-dimensional torus of sidelength $n$. 
	
	\begin{theorem}\label{lem:sp-torus}
		Let $\lambda > 0$ so that strong spatial mixing holds for all  $\bl \leq \lambda$.  Then $$\log Z_{\TT_n^d}(\lambda) = n^d p(\lambda) + O(e^{-\Omega(n)})\,.$$
	\end{theorem}

\subsection{Algorithmic results}
\label{secBlockDynamics}

Let $\Lam_n$ be the centered, axis-parallel box  $[-n,n]^d \subset \R^d$.  Fix a pair potential $\phi$ and an activity function $\bl$ supported on $\Lam_n$.  Let $Z_{\Lam_n}(\bl)$ and $\mu_{\Lam_n,\bl}$ denote the corresponding partition function and Gibbs point process.    The \emph{block dynamics} for $\mu_{\Lam_n,\bl}$ with update radius $L>0$  is a Markov chain on the space of finite point sets in $\Lam_n$ with the following update rule.  Given a configuration $X_t \subset \Lam_n$, form the configuration $X_{t+1}$ by:
\begin{enumerate}
\item Picking an update location $y \in \Lam_n$ uniformly at random.
\item Resampling the points in the ball $B_L(y)$ according to $\mu_{\Lam_n,\bl}$ with boundary conditions given by $X_t \setminus B_L(y)$.
\end{enumerate}
We say the block dynamics are \textit{local} if the update radius $L$ is bounded independent of  $n$. 

To be precise about the resampling step, let $H_{L,y}^{X_t}$ be the function on finite point sets in $B_L(y)$ defined by
\[ H_{L,y}^{X_t} (x_1, \dots , x_k) = \sum_{1 \le i < j \le k} \phi(x_i -x_j)  + \sum_{i=1}^k \sum_{x \in X_t \setminus B_L(y)} \phi(x_i - x) \, .\]
Then resampling according to $\mu_{\Lam_n,\bl}$ with boundary conditions given by $X_t \setminus B_L(y)$ means sampling a point process $Y$ on $B_L(y)$ with density $e^{-  H_{L,y}^{X_t}}$ against the Poisson process of intensity $\bl$ on $B_L(y)$ and setting $X_{t+1} = (X_t \setminus B_L(y) ) \cup Y$.

The block dynamics are reversible with respect to the  stationary distribution $\mu_{\Lam_n, \bl}$.  For $j =0 ,1, \dots$ let $\mu_j$ denote the distribution of the configuration $X_j$ under the block dynamics.  Then for any initial distribution $\mu_0$,  $\lim_{j \to \infty} \| \mu_j - \mu_{\Lam_n,\bl} \|_{TV} =0$ since the block dynamics are irreducible and aperiodic.  To measure the rate of convergence we use the mixing time:
\[ \tau_{\mathrm{mix}} (\eps) = \sup_{\mu_0} \,  \min \{ j \ge 0 : \| \mu_j - \mu_{\Lam_n,\bl} \|_{TV} < \eps\} \,,\]
where the supremum is over probability distributions on finite point sets in $\Lam_n$.

By showing that  strong spatial mixing implies fast mixing of block dynamics (Theorem~\ref{thmSSMtoMix} below, adapted from~\cite{dyer2004mixing,helmuth2020correlation}), we deduce the following theorem.
\begin{theorem}
\label{thmFastMixing}
For every finite-range, repulsive potential $\phi$ and every $\lam < e/\Delta_{\phi}$, there exists $L_0>0$ so that the block dynamics with update radius $L \ge L_0$ for the Gibbs point process with pair potential $\phi$ and activity function $\bl$ bounded by $\lam$ and supported on $\Lam_n$ has mixing time $\tau_{\mathrm{mix}}(\eps) = O( N \log (N/\eps) )$ where $N = |\Lam_n| = (2n)^d$.  
\end{theorem}	
The lower bound $L_0$ on the update radius  depends  on $\phi$ and $\lam$ and is independent of $n$.  To the best of our knowledge this is the best bound on the range of parameters for fast mixing of a local Markov chain for any repulsive point process.   In the case of the Strauss process~\cite{huber2012spatial} and hard sphere model~\cite{helmuth2020correlation}  this improves the known bounds on $\lam$ for efficient approximate sampling by a factor at least $e/2$ (though the  algorithm in~\cite{huber2012spatial} is a perfect sampler).  The restriction to boxes $\Lam_n$ is not essential and other regions can be obtained simply by changing $\bl$; some regularity in the shape of the region is necessary to obtain mixing time bounds as a function of volume (see the discussion in~\cite{helmuth2020correlation}).

Theorem~\ref{thmFastMixing} gives an efficient sampling algorithm for $\mu_{\Lam_n,\bl}$ under the assumption that we can implement a single step of the block dynamics efficiently.  In the case of the constant activity function $\lam$ (and in a model of real-valued computation such as that of~\cite{blum1989theory}, which is necessary to obtain total variation bounds for point processes), this can be done via acceptance--rejection sampling: sample a Poisson process $Y$ of intensity $\lam$ on $B_L(y)$ and accept with probability $e^{-H_{L,y}^{X_t}(Y)}$; if $Y$ is rejected, resample and repeat.  Recall that to perfectly sample a inhomogeneous Poisson process of intensity $\bl$ bounded by $\lambda_0$, we may sample a homogeneous Poisson process $Y$ of intensity $\lambda_0$ and delete each point $y \in Y$ independently with probability $1 - \bl(y)/\lambda_0$.  Thus the update step in the general case can be performed efficiently assuming constant-time query access to $\bl(y)$.

We next turn to the approximate counting problem. An $\eps$-relative approximation to $Z$ is a number $\hat Z$ so that $e^{-\eps} \hat Z \le Z \le e^{\eps} \hat Z$.   In the setting of discrete spin models, a \textit{fully polynomial-time randomized approximation scheme} (FPRAS) is an algorithm that given a graph $G$ on $N$ vertices and an error parameter $\eps>0$ outputs an $\eps$-relative approximation to $Z_G$ with probability at least $3/4$ and runs in time polynomial in $N$ and $1/\eps$.     
In our setting the volume  $|\Lam_n|$ will replace the number of vertices of a graph as the measure of the size of an instance.

Based on the sampling algorithm of Theorem~\ref{thmFastMixing},  we give an efficient approximation algorithm for the partition function when activities are bounded by $\lam < e/\Delta_{\phi}$ (assuming efficient implementation of a single step of the block dynamics). 
\begin{cor}
\label{thmApproxAlg}
For every finite-range, repulsive potential $\phi$ and every activity function $\bl$ on $\Lam_n$ bounded by $\lam < e/\Delta_{\phi}$, there is a randomized algorithm  that with probability at least $3/4$ returns an $\eps$-relative approximation to  $Z_{\Lam_n}(\bl)$.  
The algorithm runs in time $O(N^{3} \eps^{-2}  \log(N/\eps))$ where $N=|\Lam_n|$,  assuming unit cost for a single radius-$L$ block dynamics update, for  $L>0$ independent of $n$.
\end{cor}

For the special case of the hard sphere model,  Friedrich,  G{\" o}bel,  Krejca, and 
Pappik~\cite{friedrich2021spectral} recently gave an FPRAS for the partition function $Z_{\Lam_n}(\lam)$ when $\lam < e/C_{\phi}$, a more restrictive range of $\lam$ than that of Theorem~\ref{thmApproxAlg}.  Their algorithm is based on discretizing $\Lam_n$ and using a randomized algorithm to sample from a hard-core model.  The running time of their algorithm is $N^{O(1/\del^2)}$ where $\del = e/C_{\phi} - \lam$, while the  algorithm of Corollary~\ref{thmApproxAlg} runs in time $\tilde O(N^3)$, where the implied constant is a function of $\delta' = e/\Delta_{\phi} - \lam$.

We can combine our algorithmic results with the results of Section~\ref{subsecPressure} to obtain approximation algorithms for the pressure and surface pressure.  An $\eps$-additive approximation to a real number $p$ is a real number $\hat p$ so that $| p - \hat p|\le \eps$. 
\begin{theorem}\label{thm:pressure-algs}
		For every finite-range, repulsive potential $\phi$ and every $\lambda < e/\Delta_\phi$, there is a randomized algorithm that with probability at least $3/4$ returns an $\eps$-additive approximation to $p(\lambda)$.  The algorithm runs in time $O(\eps^{-2} \log^{d+1}(1/\eps))$ assuming unit cost for a single radius-$L$ block dynamics update for fixed $L > 0$. 
		
	Additionally, there is a randomized algorithm  that with probability at least $3/4$ returns an $\eps$-additive approximation to the surface pressure along a box that runs in time $O(\eps^{-2} \log^{d + 3}(1/\eps) )$.
	\end{theorem} 

Note that the $1/4$ failure probability can be made smaller than any $\del>0$ by repeating the algorithm $O(\log (1/\del))$ times and taking a median.

\subsection{Examples}
\label{secExamples}

Two important examples of point processes interacting via finite-range, repulsive pair potentials are the hard-sphere model and the Strauss process.    

\begin{example}
The hard sphere model  is defined by setting $\phi(x) = +\infty$ if $\| x\| <r$ and $0$ otherwise, for some interaction radius $r>0$.  This potential forbids configurations of points in which  any pair of points is within distance $r$, or  in other words, valid configurations are sets of centers of packings of spheres of radius $r/2$.  The hard sphere model is a long-studied toy model of a gas, and in dimension  three is expected to exhibit a crystallization phase transition in the infinite volume limit.  Recent results on sampling from the hard sphere model with rigorous running-time bounds include~\cite{hayes2014lower,helmuth2020correlation,huber2013bounds,jerrum2019perfect,kannan2003rapid,wellens2018note}. 
\end{example}

\begin{example}
The Strauss process~\cite{strauss1975model} is defined by taking  $\phi(x) = A$ if $\| x\| <r$ and $0$ otherwise, for some $r, A>0$.  Strauss proposed this as a model of clustering but Kelly and Ripley~\cite{kelly1976note} soon pointed out that the model is only well defined in the repulsive case.  The process is a soft-interacting variant of the hard sphere model: overlapping spheres are not forbidden but they are penalized.  It is used to model phenomenon such as the location of trees in a forest that exhibit anti-clustering behavior.   Previous rigorous results on sampling from the Strauss process include~\cite{huber2012spatial}.
\end{example}

\subsection{Methods}
\label{secMethods}

One ingredient we use to prove strong spatial mixing is the recursion from~\cite{michelen2020analyticity} given in Lemma~\ref{lem:recursion} below; this recursion may be viewed as a continuous analogue of Weitz's recursion for  occupation probabilities in the hard-core model~\cite{Weitz}.   This allows us to write a density $\rho_{\bl}(v)$ in terms of \emph{other} densities $\rho_{\bl'}$ for some explicit activity vectors $\bl'$.  This gives a heuristic way of computing a density $\rho_{\bl}(v)$: apply the recursion to write $\rho_{\bl}$ in terms of other densities and continue applying the recursion to those densities and so on.  
Similar to~\cite{Weitz}, this recursion implicitly defines a tree-like computational object (explored at length in~\cite{mp-CC}).    The potential-weighted connective constant provides a measure of the strength of a repulsive pair potential as it interacts with the geometry of the underlying space ($\R^d$ here, with other examples considered in~\cite{mp-CC}).  This notion was inspired by the connective constant used in~\cite{sinclair2013spatial,sinclair2017spatial} to give efficient counting and sampling algorithms for the hard-core model.  In the special case of the hard sphere model, the potential-weighted connective constant is a continuum analogue of the discrete connective constant.  

To show strong spatial mixing, we show that in the regime of $\lambda < e/\Delta_\phi$ we can truncate this computational tree at moderate depth and obtain a close approximation to the density $\rho_{\bl}(v)$.  In practice, this amounts to showing that if we have an activity function $\bl$ bounded by $\lambda < e/\Delta_\phi$, then a $k$-fold iteration of the recursion is in fact a \emph{contraction} (see Proposition \ref{prop:connective-contract}) for $k$ large enough in terms of $\lambda$.  Writing the measure of an event in terms of densities (Lemma \ref{lemProjection}) and applying Proposition \ref{prop:connective-contract} gives strong spatial mixing.

The inspiration for writing identities for the pressure and surface pressure involving densities comes from the work of Gamarnik and Katz who studied the pressure and surface pressure of discrete models (the hard-core and monomer-dimer models) using the `sequential cavity method'~\cite{gamarnik2009sequential}, a deterministic algorithm for approximating these quantities on $\mathbb Z^d$ exploiting strong spatial mixing.

\subsection{Open questions}

The algorithm we obtain via Theorem~\ref{thmFastMixing} is an approximate sampling algorithm based on a Markov chain: we run the chain for a prescribed number of steps (depending on the desired accuracy) and the final configuration has distribution guaranteed to be $\eps$-close to the target distribution in total variation distance.   On the other hand, one can ask for a sample distributed \textit{exactly} according to the target distribution.    Perfect sampling algorithms (e.g.~\cite{huber2016perfect,jerrum2019perfect,kendall1998perfect,kendall2000perfect}) accomplish this, and we say they are efficient if their expected running time is polynomial in the input size (in this case, polynomial in the volume of the box from which we are sampling).   Perfect sampling algorithms are appealing in practice since no proof of  efficiency is needed for a guarantee on their output distribution.  For discrete spin systems on graphs of subexponential growth, strong spatial mixing implies the existence of an efficient perfect  sampling algorithm~\cite{feng2019perfect}. We ask if such an implication holds for Gibbs point processes as well.
\begin{question}
Is there an efficient perfect sampling algorithm for Gibbs point processes defined by finite-range, repulsive pair potentials with $\lam <  e/\Delta_{\phi}$?
\end{question}

Regarding approximation algorithms for the partition function, both the algorithm of Corollary~\ref{thmApproxAlg} and the algorithm for the hard sphere partition function from~\cite{friedrich2021spectral} use randomness, while in the discrete setting there is a \textit{fully polynomial-time approximation scheme} (an FPTAS; an efficient approximation scheme that is deterministic)  due to Weitz for the hard-core model  on  graphs of maximum degree $\Delta$  when $\lam < \lam_c(\Delta)$, and deterministic approximation algorithms for the pressure and surface pressure in~\cite{gamarnik2009sequential}. 
\begin{question}
Is there an efficient deterministic approximation algorithm (an FPTAS) for $Z_{\Lam_n}(\lam)$ for $\lam <  e/\Delta_{\phi}$ for repulsive potentials?
\end{question}

Recently, Friedrich,  G{\H o}bel, Katzmann, Krejca, and Pappik demonstrated a deterministic algorithm for $Z_{\Lam_n}(\lam)$ that runs in quasipolynomial time \cite{friedrichAlg} for the case of hard spheres and $\lambda < e/C_\phi$.  It remains to demonstrate a FPTAS as well as a quasipolynomial time algorithm in the regime $\lam < e/\Delta_{\phi}$ for all (finite-range) repulsive potentials.

The algorithm of~\cite{gamarnik2009sequential} to approximate the pressure   gives an $\eps$ additive approximation to $p(\lam)$ with running time polynomial in $1/\eps$.  This is a similar running time to an algorithmic implementation of the cluster expansion, but the Gamarnik--Katz algorithm works for a much wider range of $\lam$, up to at least the bounds for strong spatial mixing proved in~\cite{Weitz,sinclair2017spatial}.  We ask if such an algorithm exists for Gibbs point processes.
\begin{question}
Is there an algorithm that gives an $\eps$ additive approximation to $p(\lam)$ and $\rho(\lam)$ in time polynomial in $1/\eps$  for repulsive point processes when $\lam < e/\Delta_{\phi}$?   Is there such an algorithm for \emph{any} $\lam >0$?
\end{question}

 \subsection{Outline }

In Section~\ref{secPrelim} we introduce some preliminary definitions and results about densities and $k$-point densities, including the recursive identity from~\cite{michelen2020analyticity} given in Lemma~\ref{lem:recursion}. In Section \ref{secSSM} we prove Theorem \ref{thmSSM}, showing strong spatial mixing.  In Section~\ref{secPressure}, we prove Theorem~\ref{lem:pressure-dir} as well as another useful  identity (Lemma~\ref{lem:interpolate-to-zero}) for the infinite-volume pressure.  In Section~\ref{secSurface}, we address the surface pressure and finite-volume corrections to the pressure, proving Theorems~\ref{lem:surface-pressure} and~\ref{lem:sp-torus}.  In Section~\ref{secAlgorithms} we prove Theorem~\ref{thm:pressure-algs}, our algorithmic results for the pressure and surface pressure. 
  In Appendix~\ref{secSSMtoBlock} we show  that strong spatial mixing implies fast mixing of block dynamics and prove Theorem~\ref{thmFastMixing}.  In Appendix~\ref{secApproxCount} we prove Corollary~\ref{thmApproxAlg}.


	\section{Preliminaries}
	\label{secPrelim}

	We begin with some definitions and lemmas about Gibbs point processes.  
	An \emph{activity function} $\bl:\R^d \to \R_{\geq 0}$ is an integrable function on $\R^d$; we say it is $\lambda$-bounded if $\| \bl \|_{\infty} \leq \lambda$.  The \emph{density} (or one-point density) at a point $v \in \R^d$ at activity $\bl$ is defined by~\eqref{eq:added-energy} above, but can be rewritten as the ratio of partition functions:
	 \begin{equation}\label{eq:density-def}
	\rho_{\bl}(v) = \bl(v) \frac{Z(\bl e^{-\phi(v-\cdot)})}{Z(\bl)}
	\end{equation}
	where the activity $\bl e^{-\phi(v-\cdot)}$ is the function $w \mapsto \bl(w) e^{-\phi(v-w)}$.   This follows from applying the GNZ equations; see e.g.,~\cite{ruelle1999statistical,jansen2019cluster}.

	Similarly, the $k$-point density---or $k$-point correlation function---at $(v_1,\ldots,v_k) \in (\R^d)^k$ is given by 
	\begin{equation}\label{eq:k-density}
	\rho_{\bl}(v_1,\ldots,v_k) = \bl(v_1)\cdots \bl(v_k) \frac{Z(\bl e^{-\sum_{i = 1}^k \phi(v_i - \cdot) } )}{Z(\bl)} e^{-H(v_1,\ldots,v_k)} 
	\end{equation}
	where the activity function $\bl e^{-\sum_{i = 1}^k \phi(v_i - \cdot)}$ is the function $w \mapsto \bl(w) e^{-\sum_{i = 1}^k \phi(v_i - w)}$.  The $k$-point density can also be defined in terms of the added energy of the $k$-tuple of points in a manner similar to \eqref{eq:added-energy}.

The next lemma expresses a one-point density in terms of an integral of other one-point densities. 	
	\begin{lemma}[\cite{michelen2020analyticity},  Theorem 8]
		\label{lem:recursion}
		Suppose an activity function $\bl$ is bounded and has bounded support.  Then for every $v \in \R^d$ we have 
		\begin{equation}
		\label{eqIdentity}
		\rho_{\bl}(v) = \bl (v) \exp\left(-\int_{\R^d}\rho_{\bl_{v \to w}}(w)(1 - e^{-\phi(v-w)})     \,dw  \right)\,,
		\end{equation}
where for $v, w \in \R^d$ the activity function $\bl_{v \to w}: \R^d \to \mathbb C$ is defined by
	\[ \bl_{v \to w}(x) = \begin{cases} \bl (x) e^{-\phi(v-x)} &\mbox{if } \|v -x\| < \| v -w\| \\ 
	\bl(x) &\mbox{if } \| v- x \|  \ge \| v -w\|  \,.\end{cases} \]  
	\end{lemma}
	
		While the right-hand side of \eqref{eqIdentity} may be hard to parse, it is useful to think of this lemma more abstractly: it states that  a given density $\rho_{\bl}$ may be written as a functional $F$ evaluated at \emph{other} densities $\rho_{\bl_{v \to w}}$, where $\rho_{\bl_{v \to w}}$ is pointwise upper bounded by $\rho_{\bl}$.  One could then apply this identity again to the densities $\rho_{\bl_{v \to w}}$, and then apply it again and so forth recursively.  Upon repeating this process, properties of the high iterates of this functional $F$ come into play; in~\cite{mp-CC} it was shown that after a change of variables, a sufficiently high iterate of the functional $F$ is a contraction, thus allowing us to compare densities at different activities.  This takes the form of the following proposition, which is crucial to our analysis here.
	\begin{prop}
		\label{prop:connective-contract}
		Suppose $\phi$ has range at most $r$ and let $V_k$ be defined as in~\eqref{eq:Vk-def}.  Then for every $\lambda \ge 0$, $v \in \R^d$ and activity functions $\bl,\bl' $ bounded by $\lambda$ we have  $$\left|\rho_{\bl}(v) - \rho_{\bl'}(v)\right| \leq 2 \lambda (\lambda / e)^{k/2}  \sqrt{V_k} $$
		where $k = \lfloor \dist(v,\supp(\bl - \bl')) / r \rfloor$
	\end{prop}	
	\begin{proof}
		This follows from combining Lemmas 23 and 26 from~\cite{mp-CC}.
	\end{proof}

	In order to prove Theorems~\ref{lem:pressure-dir} and~\ref{lem:surface-pressure}, we will need an identity for $\log Z(\bl)$; this may be interpreted as an analogue of a certain telescoping product formula for $1/Z$ in the discrete case employed by Gamarnik and Katz in \cite{gamarnik2009sequential}.
		\begin{lemma}[\cite{mp-CC}, Lemma 33] \label{lem:log-Z}
			Let $\bl$ be bounded and have bounded support; for any $q \in [1,\infty]$ we have  
			$$\log Z(\bl) = \int_{\R^d} \rho_{\hat{\bl}_x}(x)\,dx$$ where $$\hat{\bl}_x(s) = \begin{cases}
			0 & \text{ if } \|s \|_{q} < \| x\|_{q} \\ 
			\bl(s) & \text{ otherwise.}
			\end{cases}$$ 
	\end{lemma}

\section{Strong spatial mixing}
\label{secSSM}

Our starting point for proving strong spatial mixing is an identity for the probability of an event involving a point process in terms of the $k$-point density functions and one-point density functions with modified activity functions.  Recall the notion of projecting the law of $\mu_{\bl}$ to $\Lam \subset \R^d$ from Section~\ref{subsecSSM}.
\begin{lemma}
	\label{lemProjection}
	Suppose $A \in \mathfrak N(\Lam)$ and let $x_0 \in \Lambda$.   Then
	{\small
	\begin{align*}
	\mu_{\bl}( A_{\Lam}) &= \sum_{k \ge 0} \frac{1}{k!} \int_{\Lam^k} \one_{\{x_1, \dots, x_k \} \in A} \cdot \rho_{\bl} (x_1, \dots , x_k)   \exp \left ( - \int_{\Lam}  \rho_{\hat \bl_{x, x_1, \dots x_k}}(x) \, d x  \right ) \, dx_1, \cdots d x_k \,.
	\end{align*}}
	where 
	\[ 
	\hat \bl_{x, x_1, \dots, x_k} (y) =  \begin{cases}  0 &\mbox{if } \|y -x_0\| < \|x -x_0 \| \,\text{ and }\, y \in \Lambda \\
	\bl(y) \prod_{i=1}^k e^{- \phi(y-x_i)}  &\mbox{otherwise } 
	\end{cases}\,.
	\]
\end{lemma}
\begin{proof}
	We will apply \eqref{eq:mu-def} and let $\by = (y_1,\ldots,y_k)$ denote the points in $\Lambda$ and $\bz = (z_1,\ldots,z_j)$ the points in $\Lambda^c$, to write \begin{equation*}
		\mu_{\bl}( A_{\Lam}) = \sum_{k \geq 0} \frac{1}{k!} \int_{\Lambda^k} \one_{\by \in A } \frac{\bl(y_1)\cdots \bl(y_k) e^{-H(\by)}}{Z(\bl)}\sum_{j \geq 0} \frac{1}{j!} \int_{(\Lambda^c)^j } e^{-H(\bz)} \prod_{i = 1}^j \bl_{\by}(z_i) 		\,d\bz \,d\by
	\end{equation*}
	where we define the activity $\bl_{\by}$ via $\bl_{\by}(x) = \bl(x) e^{- \sum_{i = 1}^k \phi(x - y_i) }$.  Note that $$ \sum_{j \geq 0} \frac{1}{j!} \int_{(\Lambda^c)^j } e^{-H(\bz)} \prod_{i = 1}^j \bl_{\by}(z_i) 		\,d\bz = Z(\bl_{\by}) \mu_{\bl_{\by}}(\bX \cap \Lambda = \emptyset )\,.$$ 
	We now claim that \cite[Lemma 33]{mp-CC} shows \begin{equation}\label{eq:mu-prob-identity}
		\mu_{\bl_{\by}}(\bX \cap \Lambda = \emptyset ) =  \exp\left(-\int_{\Lambda} \rho_{\hat{\bl}_{x,y_1,\ldots,y_k}}(x)\,dx \right)\,.\end{equation}
	To see this, consider the metric $\tilde{d}$ on $\R^d$ given by $\tilde{d}(x,y) = d(x,y) + \diam(\Lambda)\one_{x,y \text{ not both in }\Lambda \text{ or }\Lambda^c}$.  This has the effect of ordering all points in $\Lambda$ before those in $\Lambda^c$ when interpolating.  Applying \cite[Lemma 33]{mp-CC} with this metric for both $\bl_y$ and $\bl_{y} \one_{\Lambda^c}$ shows \eqref{eq:mu-prob-identity}.
	
	Combining the previous displayed equations shows $$	\mu_{\bl}( A_{\Lam}) = \sum_{k \geq 0} \frac{1}{k!} \int_{\Lambda^k} \one_{\by \in A } \frac{\bl(y_1)\cdots \bl(y_k) e^{-H(\by)}Z(\bl_{\by})}{Z(\bl)} \exp \left ( - \int_{\Lam}  \rho_{\hat \bl_{x, y_1, \dots y_k}}(x) \, d x  \right ) d\by\,.$$
	
	Finally, equation \eqref{eq:k-density} states \begin{align*}
		 \frac{\bl(y_1)\cdots \bl(y_k) e^{-H(\by)}Z(\bl_{\by})}{Z(\bl)} = \rho_{\bl}(y_1,\ldots,y_k)
	\end{align*}
	which completes the proof.
\end{proof}

Lemma \ref{lemProjection} provides a concrete route towards proving strong spatial mixing: we will first compare one-point densities for different activity functions $\bl$ and $\bl'$ via Proposition~\ref{prop:connective-contract}.   We then will show that  the $k$-point densities $\rho_{\bl}(x_1,\ldots,x_k)$ can be expressed as a product of one-point densities, and conclude that if $\bl$ and $\bl'$ agree near $x_1, \dots, x_k$,  the respective $k$-point densities are close. 

Proposition \ref{prop:connective-contract} shows that provided $\lambda < e/\Delta_\phi$, the densities $\rho_{\bl}(v)$ and $\rho_{\bl'}(v)$ agree up to a small additive error provided $\bl$ and $\bl'$ agree near $v$.  It will be convenient to have a multiplicative error instead, and so we use the following easy lower bound for $\rho_{\bl}$.

\begin{lemma}\label{lem:rho-lb}
	Suppose $|\bl| \leq \lambda$ and $\phi$ is repulsive with range $r$.  Then for any $v$ we have $$\bl(v)e^{-\lambda|B_r|} \leq \rho_{\bl}(v)\leq \bl(v)\,.$$
\end{lemma}
\begin{proof}
	By \eqref{eq:added-energy} we have $$\rho_{\bl}(v) = \bl(v) \E_{\bl} e^{-H_v(\bX)}\,.$$ 
	Since $H_v(\bX) \geq 0$, the upper bound follows.  By Poisson domination, $\bX$ is stochastically dominated by a Poisson process of intensity $\lambda$; it follows that $\P_{\bl}(\bX\cap B_r = \emptyset) \geq e^{-\lambda  |B_r|}$.  On this event, the expectation is equal to $1$, and so the lower bound follows. 
\end{proof}

Combining this lemma with Proposition \ref{prop:connective-contract} gives exponential decay of dependence of densities on boundary conditions.  
\begin{cor}\label{cor:Vk-contract}
	Suppose $\phi$ is repulsive and of range $r$.  
	For every $0 \leq \lambda < e/\Delta_\phi$ there are constants $a,b > 0$ so that, for all  $\lam$-bounded activity functions $\bl,\bl'$ with $\bl(v) = \bl'(v) $ we have 
	$$\rho_{\bl}(v) \leq  \rho_{\bl'}(v)(1 +  a e^{-b \cdot \dist(v,\supp(\bl - \bl'))})\,.$$
\end{cor}
\begin{proof}
	Let $\eps \in (0,1)$ so that $\lambda(1 + \eps) < e/\Delta_\phi$.  By the definition of $\Delta_\phi$, there exists a $M$ so that for all $k \geq M$ we have $V_k \leq (\Delta_\phi(1 + \eps/2))^k$.  If $\lfloor\dist (v,\supp(\bl - \bl'))/r\rfloor \geq M$, then Proposition \ref{prop:connective-contract} implies $$|\rho_{\bl}(v) - \rho_{\bl'}(v)| \leq 2 \lambda \left(\frac{1 + \eps/2}{1 + \eps} \right)^{k/2} \leq \lambda\cdot 2^{r+1}\cdot  e^{- b \cdot \dist(v,\supp(\bl-\bl'))}$$
	where we have set $e^{-b} = \left(\frac{1 + \eps/2}{1 + \eps}\right)^{r/2}$.  Using the lower bound guaranteed by Lemma \ref{lem:rho-lb} completes the proof.
\end{proof}

	We next obtain the analogous statement for $k$-point densities.
\begin{cor}
	\label{corExpkpoint}
	In the context of Corollary \ref{cor:Vk-contract}, we have
	$$ \rho_{\bl}(v_1,\ldots,v_k) \leq \rho_{\bl'}(v_1,\ldots,v_k)(1 + a e^{-b \cdot \dist(v,\supp(\bl - \bl'))  })^k\,.$$
\end{cor}
\begin{proof}
	Using the definition of $ \rho_{\bl}(v_1, \dots, v_k) $ from~\eqref{eq:k-density}, we can write
	\begin{align*}
	\rho_{\bl}(v_1, \dots, v_k) &= \bl(v_1)\cdots \bl(v_k) \frac{Z(\bl e^{-  \sum_{i=1}^k\phi(v_i - \cdot)})}{Z(\bl)} e^{-H(v_1,\ldots,v_k)} \\
	&= \prod_{j = 1}^k \bl(v_j) e^{- \sum_{i < j} \phi(v_i - v_j) } \frac{Z(\bl e^{-\sum_{i = 1}^{j} \phi(v_i - \cdot)  } )}{Z(\bl e^{-\sum_{i = 1}^{j-1} \phi(v_i - \cdot)  } )} \\
	&= \prod_{j = 1}^k \rho_{\bl e^{- \sum_{i = 1}^{j-1} \phi(v_i - \cdot)}}(v_j)\,.
	\end{align*}
	Applying Corollary \ref{cor:Vk-contract} completes the proof.
\end{proof}

Corollary~\ref{corExpkpoint} will bound the $k$-point densities appearing in Lemma \ref{lemProjection}; the following bound will handle the exponential term.

\begin{cor} In the context of Corollary \ref{cor:Vk-contract}, further assume that $\lambda |\Lambda| a e^{-b t} < 1$. Then
	$$\exp\left(- \int_{\Lambda} \rho_{\hat{\bl}_{x,x_1,\ldots,x_k}}(x)\,dx\right) \leq \exp\left(- \int_{\Lambda} \rho_{\hat{\bl}_{x,x_1,\ldots,x_k}'}(x)\,dx\right) (1 + 2 a \lambda |\Lambda| e^{-b t})\,.$$
\end{cor}
\begin{proof}
	This follows from the inequality $e^{-y} \leq e^{-y(1 + \eps)}(1 + 2y\eps)$ for $y \eps < 1$ along with Corollary \ref{cor:Vk-contract}.
\end{proof}

With these ingredients in place, we now deduce strong spatial mixing.	
	\begin{proof}[Proof of Theorem~\ref{thmSSM}]
		Let $a$ and $b$ be the constants guaranteed by  Corollary~\ref{cor:Vk-contract}.  We aim to show that for the choice of $\alpha = 8 a \max\{\lambda,1\}$ and $\beta = b$ we have
		\begin{equation}
		 \| \mu_{\bl} - \mu_{\bl'} \|_{\Lam} \le \alpha | \Lam|  e^{-\beta t } \, 
		 \end{equation}
		 where we set $t= \mathrm{dist}(\Lam, \supp(\bl-\bl'))$. 
 		We may assume that $\alpha |\Lambda| e^{-\beta t} < 1$, as otherwise the inequality is trivial; in this case, we also have $\lambda|\Lambda| a e^{-\beta t} < 1$ by the choice of $\alpha$.  By Lemma~\ref{lemProjection} and Corollary~\ref{corExpkpoint} we have 
		{\small
		  \begin{align*}
		\|& \mu_{\bl}  - \mu_{\bl'} \|_{\Lam} =  \sup_{A \in \mathfrak N(\Lam) } | \mu_{\bl}( A_{\Lam}) - \mu_{\bl'}( A_{\Lam}) | \\
		&\le  \sum_{k \ge 1} \frac{1}{k!} \int_{\Lam^k} \bigg|   \rho_{\bl} (\bx)   \exp \left ( - \int_{\Lam}  \rho_{\hat \bl_{x, x_1, \dots x_k}}(x) \, d x  \right ) 
		-  \rho_{\bl'} (\bx)   \exp \left ( - \int_{\Lam}  \rho_{\hat \bl'_{x, x_1, \dots x_k}}(x) \, d x  \right ) \bigg| \, d \bx \,.
	\end{align*}
	}
	For each $\bx$ apply Corollary~\ref{corExpkpoint} and write $\rho_{\bl'}(\bx) = \rho_{\bl}(\bx)(1 + E)^k$ for $|E| \leq ae^{-\beta t}$ and bound $|\rho_{\hat \bl_{x, x_1, \dots x_k}}(x) -\rho_{\hat \bl'_{x, x_1, \dots x_k}}(x) | \leq \lambda (1 + a e^{-b t})$.  We then see that the previous displayed equation is at most 
	{\small
	\begin{align*} 
		&\leq \sum_{k \geq 1} \frac{1}{k!} \int_{\Lambda^k} \rho_{\bl} (\bx)   \exp \left ( - \int_{\Lam}  \rho_{\hat \bl_{x, x_1, \dots x_k}}(x) \, d x  \right )d \bx ((1 + ae^{ - \beta t})^k(1 + 2 a |\Lambda| e^{-\beta t}) - 1) \\
		&\leq (1 + 2a |\Lambda| e^{-\beta t}) \exp(\lambda |\Lambda| a e^{-\beta t}  ) - 1\,.
		\end{align*}}
where on the last line we use Poisson domination and the fact that if $X$ is a Poisson random variable we have $\E r^X = \exp(\E X(r-1) )$ for $r \in \R$.  Using the bound $e^x \leq 1 + 2x$ for $x \in [0,1]$ bounds \begin{align*}\| \mu_{\bl}  - \mu_{\bl'} \|_{\Lam} &\leq (1 + 2a |\Lambda| e^{-\beta t}) \exp(\lambda a |\Lambda| e^{-\beta t}  ) - 1 \\ 
		&\leq (1 + 2a|\Lambda| e^{-\beta t})(1 + 2\lambda a |\Lambda| e^{-\beta t}) - 1 \\
		&\leq \alpha |\Lambda| e^{-\beta t}\,,
		\end{align*}
		as desired.
	\end{proof}

	\section{Identities for the pressure}
	\label{secPressure}
	
	Here we derive two identities for the pressure using two different interpolations between the constant  activity functions $\bl \equiv 0$ and  $\bl \equiv \lambda$.  
	
	\subsection{Interpolating left-to-right: Proof of Theorem \ref{lem:pressure-dir}}
	
	First we use a consequence of Lemma \ref{lem:log-Z}.

	\begin{cor}\label{cor:left-to-right}
		Let $\Lambda_n = [-n,n]^d$ be the solid hyper-cube of side length $2n$.  Define the activity $\bl_t$ by $\bl_t(x) = \lambda \one\{x_1 \geq t\}$. Then $$\log Z_{\Lambda_n}(\lambda) = \int_{\Lambda_n} \rho_{\bl_{x_1}}(x) \,dx\,. $$
	\end{cor}
	\begin{proof}
		By translation invariance of $\phi$, we may consider $\Lambda_n = n \ve_1 + [0,2n] \times [-n,n]^{d-1}$.  Applying Lemma \ref{lem:log-Z} with $q = +\infty$, we see that in fact $\hat{\bl}_x(s) = \lambda\one \{s \in \Lambda_n\} \one \{s_1 \geq x_1  \}$.  Shifting back to the origin and changing variables  completes the proof.  
	\end{proof}
	Theorem~\ref{lem:pressure-dir} now follows  from Corollary \ref{cor:left-to-right} along with strong spatial mixing.  
		\begin{proof}[Proof of Theorem~\ref{lem:pressure-dir}]
		Note that by applying a rotation, we may assume without loss of generality that $\vu = \ve_1$.  Define the set $\Lambda_n' := \{x \in \Lambda_n : d(x, \Lambda_n^c) \geq \log^2 n \}$, and note that $|\Lambda_n \setminus \Lambda_n'|/ |\Lam_n| = O(1/\log^2(n))$ and so Corollary \ref{cor:left-to-right} shows \begin{equation} \label{eq:Lambdap}
		\log Z_{\Lambda_n}(\lambda) = \int_{\Lambda_n'} \rho_{\bl_{x_1}}(x)\,dx + O(|\Lam_n|/\log^2 n)\,.
		\end{equation}
		
		By strong spatial mixing  and translational invariance, we note that for $x \in \Lambda_n'$ we have \begin{equation} \label{eq:mixing-Lambdap}
		|\rho_{\bl_{x_1}}(x) - \rho_{\bl_{\ve_1}}(0)| \leq e^{-\Omega( \log^2 n)}\,. 
		\end{equation}
		
		Combining \eqref{eq:Lambdap} and \eqref{eq:mixing-Lambdap} completes the proof. 
	\end{proof}

\subsection{Interpolating uniformly}
In this section we deduce another identity for the pressure that stems from interpolating between activity $0$ and activity $\lambda$.  To prove it, we first recall a basic fact about the logarithmic derivative of the partition function.  Throughout this section, we write $\rho_{\Lambda,\lambda}(v)$ for $\rho_{\bl}(v)$ with the choice of activity $\bl(x) = \lambda \one_{x \in \Lambda}$. 

\begin{fact} \label{fact:expectation}
	For each $\lambda > 0$ we have $$\frac{\lambda Z_{\Lambda}'(\lambda)}{Z_{\Lambda}(\lambda)} = \int_{\Lambda} \rho_{\Lambda,\lambda}(v)\,dv\,.$$
\end{fact}
\begin{proof}
	The left-hand side is equal to the expected number of points of the Gibbs point process in $\Lambda$ at activity $\lambda$, which is equal to the right-hand side (see, for instance, the GNZ equations \cite{ruelle1999statistical}).
\end{proof}

From here, our interpolation statement  follows  from the fundamental theorem of calculus.

\begin{lemma}\label{lem:interpolate-to-zero}
	Let $\lambda > 0$.  Then $$\log Z_{\Lambda}(\lam) =  \int_0^1  \int_{\Lambda} \frac{\rho_{\Lambda, t \lambda}(v)}{t}\,dv \,dt\,.$$
\end{lemma}
\begin{proof}
	By the fundamental theorem of calculus and Fact \ref{fact:expectation} we have $$\log Z_{\Lambda}(\lambda) = \int_0^1 \frac{d}{dt} \log Z_{\Lambda}(t\lambda) \,dt = \int_0^1 \frac{\lambda Z_{\Lambda}'(t\lambda)}{Z_{\Lambda}(t\lambda)}\,dt = \int_0^1  \int_{\Lambda} \frac{\rho_{\Lambda, t \lambda}(v)}{t}\,dv \,dt\,.$$
\end{proof}

We note that by the identity \eqref{eq:added-energy}, we may write $\frac{\rho_{\Lambda,t\lambda}(v)}{t} = \lambda \E_{\Lambda,t\lambda} e^{-H_v(\bX)}$, and so this ratio is in fact deterministically bounded and given as $\lambda$ times an expectation of a random variable that is bounded by $1$.  In particular, strong spatial mixing will tell us that if we take a sequence of regions $\Lambda_n \uparrow \R^d$, then this quantity converges.  This provides us with an alternate form for the pressure.

\begin{lemma}\label{lem:pressure-second-identity}
	Let $\lambda > 0$ so that we have strong spatial mixing for all $\bl \leq \lambda$.  Then $$p(\lambda) = \int_0^1 \frac{\rho_{t\lambda}(0)}{t}\,dt\,.$$
\end{lemma}
\begin{proof}
	By Lemma \ref{lem:interpolate-to-zero}, we may write $$\frac{1}{|\Lambda_n|} \log Z_{\Lambda_n}(\lambda) = \int_0^1  \frac{1}{|\Lambda_n|}\int_{\Lambda_n}\frac{\rho_{\Lambda_n, t \lambda}(v)}{t}\,dv \,dt\,.$$
	Setting $\Lambda_n' = \{v \in \Lambda_n : d(v,\partial \Lambda_n) \geq \log^2 n  \}$, we see that $|\Lambda_n'|/|\Lambda_n| = 1 + o(1)$.  If we write $$\rho_{\Lambda_n,t\lambda}(v) = t\lambda \E_{\Lambda_n,t\lambda} e^{-H_v(\bX)}\,,$$ then we see that strong spatial mixing implies that for all $v \in \Lambda_n'$ and $t \in [0,1]$ we have $$\left|\frac{\rho_{\Lambda_n,t\lambda}(v)}{t} -  \frac{\rho_{t\lambda}(v)}{t} \right| = O(e^{-\Omega(\log^2 n)})$$
	completing the proof.
\end{proof}

	\section{Surface pressure and finite-volume corrections}
	\label{secSurface}

	\subsection{Exponential convergence to the pressure on the torus}
	
	Here we prove Theorem~\ref{lem:sp-torus}.  Recall that $\TT_n^d =  \R^d / (n \Z)^d$ is the $d$-dimensional torus of sidelength $n$, and so in particular we have $| \TT_n^d| = n^d$.

	\begin{proof}[Proof of Theorem~\ref{lem:sp-torus}]
		By Lemma \ref{lem:interpolate-to-zero} and translation invariance of the torus we have $$\log Z_{\TT_n^d}(\lam) =  n^d\int_0^1  \frac{\rho_{\TT_n^d, t \lambda}(0)}{t}\,dt\,.$$  By strong spatial mixing, we have 
{\small \begin{align*}
			 \left|\frac{\rho_{\TT_n^d, t \lambda}(0)}{t} - \frac{\rho_{\R^d, t \lambda}(0)}{t}  \right| &\leq \left|\frac{\rho_{\TT_n^d, t \lambda}(0)}{t} - \frac{\rho_{[-n/4,n/4]^d, t \lambda}(0)}{t}  \right| + \left|\frac{\rho_{[-n/4,n/4]^d, t \lambda}(0)}{t} - \frac{\rho_{ t \lambda}(0)}{t}  \right| \\
		 &= O(e^{- \Omega(n) }  )\,.
		 \qedhere
		 \end{align*} } 
	\end{proof}

We may use Theorem~\ref{lem:sp-torus} to deduce a surprising identity that will be useful for simplifying a formula for the surface pressure along a cube.  For a unit vector $\vu \in \S^{d-1}$, activity $\lam$ and $t \geq 0$, define the activity function $\bl_{\vu,t}$ by $\bl_{\vu,t}(x) = \lam\one\{ \langle \vu,x \rangle \in [0,t]  \}$.  We recall that we have defined $\bl_{\vu}$ by $\bl_{\vu}(x) = \lam \one\{ \langle \vu,x\rangle \geq 0 \}$.

\begin{lemma}\label{lem:torus-alt}
	Suppose strong spatial mixing holds for all $\bl \leq \lambda$.  Then for any vector $\vu \in \S^{d-1}$ we have $$\int_{0}^\infty (\rho_{\bl_{\vu,t}}(0) - \rho_{\bl_{\vu}}(0))\,dt = 0\,.$$
\end{lemma}
\begin{proof}
	Apply a rotation so that we may assume without loss of generality that $\vu = \ve_1$.
	The strategy will be to compute $\log Z_{\TT_n^d}(\lambda) - n^d p(\lambda)$ in two ways: first by Theorem~\ref{lem:sp-torus}, and second by using Lemma \ref{lem:log-Z}.  In particular, we may view $\TT_n^d = [-n/2,n/2)^{d}$ as a product space of circles and take the lexicographical ordering.  Then, if we apply Theorem~\ref{lem:sp-torus} we have $\log Z_{\TT_n^d}(\lambda) = \int_{\TT_n^d} \rho_{\bl_x}(x)\,dx$
	where we write $\bl_x$ for the activity $\bl_x(y) = \lambda\{ |y_1| \geq |x_1| \}$ with $y,x$ written as elements of $[-n/2,n/2)^d$.  By translation invariance and symmetry, we have  $$\int_{\TT_n^d} \rho_{\bl_x}(x) = 2n^{d-1} \int_0^{n/2}\rho_{\bl_{(t,0,\ldots,0)} }((t,0,\dots,0)) \,dt\,.$$
	
	By strong spatial mixing, we have $|\rho_{\bl_{(t,0,\ldots,0)} }((t,0,\dots,0)) - \rho_{\blambda_{\ve_1,n - 2t}}(0)| =O(e^{-\Omega(n)})$.  We thus have 
	\begin{equation}
		\log Z_{\TT_n^d}(\lambda) - n^d p(\lambda) = n^{d-1} \int_{0}^n (\rho_{\bl_{\ve_1,t}}(0) - \rho_{\bl_{\ve_1}}(0))\,dt + O(e^{-\Omega(n)})\,.
	\end{equation}
	Dividing by $n^{d-1}$, taking $n \to \infty$ and comparing to Theorem~\ref{lem:sp-torus} completes the proof.
\end{proof}

	\subsection{Surface pressure along the sphere}
	We will be able to compute the surface pressure along a sphere as well.  Before doing so, we require some basic preliminaries.
	
	First, we will need a basic continuity statement saying that if $\bl$ and $\bl'$ are activities that only disagree on a set of small volume then $\rho_{\bl} \approx \rho_{\bl'}$.  For this, we use a basic fact from \cite{michelen2020analyticity}:
	
	\begin{lemma}\label{lem:UC}
		Let $\bl$ and $\bl'$ be two activity functions bounded by $\lambda$ and set $\delta = \int |\bl(x) - \bl'(x)| \,dx$. Then $|Z(\bl) - Z(\bl')| \leq  e^{\lambda^2} \delta$.
	\end{lemma}
	
	From here, we obtain that $\rho_{\bl}$ is Lipschitz in $\bl$ for bounded $\bl$.  
	
	\begin{cor}\label{cor:compare}
		Let $\bl$ and $\bl'$ be two activity functions bounded by $\lambda$ and set $\delta = \int |\bl(x) - \bl'(x)| \,dx$.  If $v$ satisfies $\bl(v) = \bl'(v)$ then $$|\rho_{\bl}(v) - \rho_{\bl'}(v)| \leq 2 \lambda e^{\lambda^2 }\delta  \,.$$
	\end{cor}
	\begin{proof}
		Recall that $$\left|\rho_{\bl}(v) - \rho_{\bl'}(v)\right| = \bl(v) \left| \frac{Z(\bl e^{-\phi(v - \cdot)})}{Z(\bl)} -  \frac{Z(\bl' e^{-\phi(v - \cdot)})}{Z(\bl')}\right|\,.$$  Bounding the above difference by  $$\frac{|Z(\bl e^{-\phi(v - \cdot)}) - Z(\bl' e^{-\phi(v - \cdot)}) |}{Z(\bl)} + \frac{Z(\bl' e^{-\phi(v - \cdot)})}{Z(\bl')}\frac{|Z(\bl) - Z(\bl') | }{Z(\bl)} \leq 2 \delta e^{\lambda^2} $$
		using Lemma \ref{lem:UC} completes the proof.
	\end{proof}
	
	To apply Corollary~\ref{cor:compare}, we will need to compare volumes of portions of a spherical cap to a portion of a half-space (see Figure \ref{figure}). 	 We prove a basic volume bound that will be useful for this comparison.

	\begin{lemma}\label{lem:geometry}
		For $R \geq m >0$, let $B_R$ denote the ball in $\R^d$ of radius $R$ centered at $(R,0,\ldots,0)$ and define the cylinder  $S := \{ x \in \R^d : \| (0,x_2,\ldots,x_d) \| \leq m, x_1 \in[0,R] \}$.  Then there is a constant $C_d > 0$ so that $$| S \setminus B_R| \leq C_d m^{d+1} R^{-1} \,.$$
	\end{lemma}
	\begin{proof}
		We will compute this volume in spherical coordinates.  In particular, we will compute the volume by integrating the $d-1$ dimensional volume of the cylindrical shells $S_t := \{ x \in S : \|(0,x_2,\ldots,x_d)\| = t \}$.  Letting $C_d$ denote the $d - 2$-dimensional Lebesgue measure of the surface area of the unit $d-1$-dimensional ball, we note that $S_t$ has a base of circumference $C_d t^{d-2}$ and height $R - \sqrt{R^2 -t^2}$.  Thus we may bound \begin{align*}
			|S \setminus B_R| &= \int_{0}^m C_d t^{d-2} (R  -\sqrt{R^2 - t^2}) \,dt = C_d R \int_{0}^m  t^{d-2}(1 -  \sqrt{1 - (t/R)^2})\,dt \\
			&\leq C_d R^{-1} \int_0^m t^d \,dt = \frac{C_d m^{d+1}}{R(d+1)}
		\end{align*}
		where in the inequality we used the bound $1 - \sqrt{1 - x} \leq x$ for $x \in [0,1]$.
	\end{proof}

	\begin{figure}[!ht]
		\centering
		\includegraphics[width=2in]{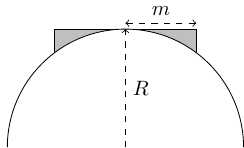}
		\caption{A volume calculation computed in Lemma \ref{lem:geometry}.}
		\label{figure}
	\end{figure}

	From here, we will find the surface pressure along a sphere.  
	\begin{prop}\label{pr:SP-sphere}
		Suppose strong spatial mixing holds for all $\bl \leq \lambda$ and let $B_n$ denote the ball of radius $n$ in $\R^d$.  Then 
	\begin{align*} \sp_B(\lambda) &:= \lim_{n\to\infty}\frac{\log Z_{B_n}(\lambda ) - |B_n| p(\lambda)}{\SA(B_n)} \\
		&= \frac{1}{|\S^{d-1}|}\int_{\vu \in \S^{d-1}} \int_{0}^\infty \int_0^1 \frac{\rho_{t\bl_{\vu}}(s\vu) - \rho_{t\lambda}(0)}{t} \,dt\,ds\,d\vu\,.
	\end{align*}
	\end{prop}

	\begin{proof}
		Set $B_n' = \{x \in B_n : \dist(x,\partial B_n) \leq \log^2n  \}$ and write \begin{align}\log Z_{B_n}(\lambda) - & |B_n| p(\lambda) \nonumber \\
			&= \int_{B_n \setminus B_n'}  \int_0^1  \frac{\rho_{B_n, t \lambda}(v) - \rho_{t\lambda}(0)}{t} \,dt\,dv+ \int_{B_n'}  \int_0^1  \frac{\rho_{B_n, t \lambda}(v)- \rho_{t\lambda}(0)}{t} \,dt\,dv \label{eq:sphere-inner-outer}
		\end{align} note that by strong spatial mixing the former integral is $O(e^{-\Omega(\log^2 n)})$. 
	
		Change to spherical coordinates to rewrite the latter integral as 
		\begin{align} 
			\int_{B_n'}  &\int_0^1  \frac{\rho_{B_n, t \lambda}(v) - \rho_{t\lambda}(0)}{t} \,dt\,dv \nonumber \\
			&= \int_{\vu \in \S^{d-1}} \int_{0}^{\log^2 n}(n-s)^{d-1} \int_0^1  \frac{\rho_{B_n, t \lambda}((n-s)\vu) - \rho_{t\lambda}(0)}{t} \,dt\,ds\,d\vu\,. \label{eq:spherical-coords}
		\end{align}
	
		Fix a given $\vu$ and $s \in [0,\log^2 n]$.  Set $U$ to be collection of points in $B_n$ so that the distance to the line $L_{\vu} := \{ r \vu : r \in \R \}$ is at most $\log^2 n$.  By strong spatial mixing  we have \begin{equation}\label{eq:swap-to-neighborhood}
			\left|\frac{\rho_{B_n, t \lambda}((n-s)\vu) - \rho_{U,t\lambda}((n-s)\vu)}{t}\right| = O(e^{-\Omega(\log^2 n)})\,.
		\end{equation}
		
		For the same vector $\vu$, let $U'$ be the cylinder defined by  $$U' := \{ x \in \R^d : \langle x, \vu \rangle \in [n - \log^2 n,n], d(x,L_{\vu}) \leq \log^2 n  \}\,.$$  By Corollary \ref{cor:compare}, Lemma \ref{lem:geometry} and strong spatial mixing, for $s \in [0,\log^2 n]$ we have \begin{equation}\label{eq:swap-to-cylinder}
			\left|\frac{\rho_{U,t\lambda}((n-s)\vu) - \rho_{U',t\lambda}((n-s)\vu)}{t} \right| = O(\log^{2d}n / n)\,.\end{equation}  By strong spatial mixing and translational invariance, for $s \in [0,\log^2 n]$ we additionally have 
		\begin{equation}\label{eq:expand-cylinder}
			\left|\frac{\rho_{U',t\lambda}((n-s)\vu) - \rho_{t\blambda_{\vu}}(s\vu)}{t}\right| = O(e^{-\Omega(\log^2 n)})\,.
		\end{equation}
		
		Combining lines \eqref{eq:sphere-inner-outer},\eqref{eq:spherical-coords}, \eqref{eq:swap-to-neighborhood}, \eqref{eq:swap-to-cylinder}, \eqref{eq:expand-cylinder} shows  \begin{align*}
			\log Z_{B_n}(\lambda) -  |B_n| p(\lambda) &= \int_{\vu \in \S^{d-1}} \int_0^{\log^2 n} (n - s)^{d - 1} \int_0^1 \frac{\rho_{t\blambda_{\vu}}(s\vu) - \rho_{t\lambda}(0)}{t}\,dt\,ds\,d\vu \\
			&\qquad+ O(n^{d-2} \cdot\log^{2d + 2} n )\,.\end{align*}
		Dividing by $\SA(B_n) = n^{d-1} |\S^{d-1}|$ completes the proof.
	\end{proof}

	\subsection{Surface pressure along a convex polytope}
	
	By a similar approach to Proposition \ref{pr:SP-sphere}, we will be able to find the surface pressure along any convex polytope.  
	
	For our setting, consider a convex polytope $\cP \subset \R^d$ of positive volume and assume without loss of generality that $0 \in \cP$.  Let $\cF$ denote the collection of the faces of its boundary.  For each face $F \in \cF$ let $|F|$ denote its surface area, $\vn_F$ denote its outward pointing normal vector, and $r_F$ its distance to the origin.   Then we note that we may change coordinates so that for any integrable function $g : \cP\to \R$ we have  \begin{equation}\label{eq:polytope-int}
		\int_{\cP} g(x)\,dx = \sum_{F \in \cF} r_F \int_0^1 \int_{F} s^{d-1} g(sy)\,dy\,ds\,.
	\end{equation}

	An identity for the surface pressure will follow similarly to Proposition \ref{pr:SP-sphere}.  
	
	\begin{prop}\label{pr:polytope} Let $\cP \subset \R^d$ be a convex polytope of positive volume.  Then if strong spatial mixing holds for all $\bl \leq \lambda$ then we have  
		\begin{align*}\sp_{\cP}(\lambda) &:= \lim_{n \to \infty} \frac{\log Z_{n \cP}(\lambda) - n^d |\cP| p(\lambda) }{\SA(n \cP)}  \\
			&= \frac{1}{\SA(\cP)}\sum_{F \in \cF} |F| \int_0^\infty \int_{0}^1 \frac{\rho_{t \blambda_{{\vn}_F}}(s\vn_F) - \rho_{t\lambda}(0) }{t} \,dt \,ds\,. 
		\end{align*}
	\end{prop}
	
Note that Theorem~\ref{lem:surface-pressure} follows immediately by taking $\cP$ to be a box.	
	\begin{proof}[Proof of Proposition~\ref{pr:polytope}]
		Set $\cP_n = n \cP$.  Apply strong spatial mixing and Lemma \ref{lem:interpolate-to-zero} in the coordinates given by \eqref{eq:polytope-int} to see 
		\begin{align}
			&\log Z_{\cP_n}(\lambda) - n^d |\cP| p(\lambda) \nonumber \\
			&= \sum_{F\in \cF} r_F \int_{n - \log^2n}^n s^{d-1} \int_{F} \int_{0}^1 \frac{\rho_{\cP_n,t\lambda}(sy) - \rho_{t\lambda}(0)}{t} \,dt\,dy\,ds + O(e^{-\Omega(\log^2 n)}) \label{eq:polytope-break-up}\,.
		\end{align}
	For a given $F$, set $F' = \{y \in F : d(y,\partial F) \geq \log^2 n / n \}$ where $\partial F$ is the $(d-2)$-dimensional polytope that gives the boundary of $F$.  Then we may replace the integral over each $F$ in the above with an integral over $F'$ and introduce a total error of size at most $O(n^{d-2} \log^2 n)$.  By strong spatial mixing, for each point $y \in F'$ and $s \in [n - \log^2n ,n]$ we have $$\left|\frac{\rho_{\cP_n,t\lambda}(sy)}{t} - \frac{\rho_{t \blambda_{{\vn}_F}}(r_F(n-s) \vn_F) }{t}\right| = O(e^{-\Omega(\log^2 n)})\,.$$		
	
	Combining with equation \eqref{eq:polytope-break-up} and changing coordinates $r_F(n - s) \mapsto s'$ completes the proof. 
	\end{proof}

	As an immediate corollary, we see that if $\phi$ is spherically symmetric then the surface pressure along a sphere is equal to the surface pressure along any convex polytope.  
	
	\begin{cor}
		Let $S$ be either a sphere or a convex polytope of positive volume.  If strong spatial mixing holds for all $\bl \leq \lambda$ then $$\sp_S(\lambda) = \lim_{n \to \infty} \frac{\log Z_{nS}(\lambda) -  |nS| p(\lambda)}{\SA(nS)} = \int_0^\infty \int_0^1 \frac{\rho_{t\blambda_{\ve_1}}(s\ve_1) - \rho_{t\lambda}(0)}{t}\,dt\,ds\,.$$
	\end{cor}
	
	\subsection{Another surface pressure identity for the box}
	Proposition \ref{pr:polytope} provides an identity for the surface pressure of a box using Lemma \ref{lem:interpolate-to-zero}.  It is also possible to use the identity in Corollary \ref{cor:left-to-right}, which is what we do in this subsection.   Further, the resulting identity will be the simplest for algorithms and indeed provide another perspective on the surface pressure. 
	
	For two unit vectors $\vu,\vec{v} \in \S^{d-1}$ and $\lam$, define the activity $\bl_{\vu,\vec{v}}$ by   
	$$\bl_{\vu,\vec{v}}(x) = \bl(x) \one\{ \langle x,\vu \rangle \geq 0, \langle x, \vec{v} \rangle \geq 0 \}\,. $$	
	\begin{lemma}\label{lem:sp-box}
		Suppose strong spatial mixing holds for all $\bl \leq \lambda$.  Then 
		$$ \sp_{\Lambda}(\lambda) = \frac{1}{d}\sum_{j = 2}^d\int_{0}^{\infty}\left(\rho_{\bl_{\ve_1,\ve_j}}(t \ve_j) +  \rho_{\bl_{\ve_1,-\ve_j}}(-t \ve_j) - 2\rho_{\blambda_{\ve_1}}(0)   \right) \,dt\,.$$
In the case that $\phi$ is spherically symmetric, then the right-hand side may be written as 	
\begin{equation}\label{eq:surf-pressure-ss}
\left(1 - \frac{1}{d}\right)\int_{ 0}^{\infty} (\rho_{\bl_{\ve_1,\ve_2}}(t\ve_2) - \rho_{\bl_{\ve_1}}(0))\,dt \,.
\end{equation}	
	\end{lemma}
	\begin{proof}
		By Corollary \ref{cor:left-to-right}, we have $$\log Z_{\Lambda_n}(\lambda) - n p(\lambda) = \int_{[-n,n]^{d} } (\rho_{\bl_{x_1}}(x) - \rho_{\blambda_{\ve_1}}(0)) \,dx\,. $$
		Set $\Lambda^{\circ} = \{ x \in \Lambda_n : \|x\|_{\infty} \leq n - \log^2 n\}$ and note that by strong spatial mixing we have that $|\rho_{\bl_{x_1}}(x)- \rho_{\blambda_{\ve_1}}(0)| = O(e^{-\Omega(\log^2 n)})$ for all $x \in \Lambda^{\circ}$. 
		Thus if we set $\Lambda' = \Lambda_n \setminus \Lambda^{\circ}$ then we have \begin{equation}  \label{eq:sp-lp}
			\log Z_{\Lambda_n}(\lambda) - n p(\lambda) = \int_{\Lambda' } (\rho_{\bl_{x_1}}(x) - \rho_{\blambda_{\ve_1}}(0)) \,dx + O(e^{-\Omega(\log^2 n)})\,.
		\end{equation}
	
	For each $j \in [d]$, define the sets $\Lambda_+^{(j)}$ and $\Lambda_-^{(j)}$ via $$\Lambda_{\pm}^{(j)} = \{ x \in \Lambda_n :  |x_j - ( \pm n)| \leq \log^2 n \text{ and } |x_i| \leq n - \log^2 n \text{ for all } i \neq j\}$$
		and set $\Lambda^{(j)} = \Lambda_+^{(j)} \cup \Lambda_-^{(j)}$.  We then note that the $\Lambda^{(j)}$ are disjoint and $|\Lambda'| = \sum_j |\Lambda^{(j)}| + O_d(n^{-(d-2)}\log^4 n )$.  We may thus rewrite \eqref{eq:sp-lp} to see \begin{equation*}
				\log Z_{\Lambda_n}(\lambda) - n p(\lambda) = \sum_{j = 1}^d \left(\int_{\Lambda_+^{(j)} \cup \Lambda_-^{(j)} } (\rho_{\bl_{x_1}}(x) - \rho_{\blambda_{\ve_1}}(0)) \,dx   \right) + o(n^{d-1})\,.
		\end{equation*}
	
	For $j \neq 1$, integrate over all variables aside from $x_j$ and apply strong spatial mixing to see \begin{align*}
		&\int_{\Lambda_+^{(j)} \cup \Lambda_-^{(j)} }(\rho_{\bl_{x_1}}(x) - \rho_{\blambda_{\ve_1}}(0)) \,dx \\
		&\quad= (2n)^{d-1} \int_{0}^{\log^2 n}(\rho_{\bl_{\ve_1,\ve_j}}(t \ve_j) +  (\rho_{\bl_{\ve_1,-\ve_j}}(-t \ve_j) - 2\rho_{\blambda_{\ve_1}}(0))   ) \,dt + o(n^{d-1})\,.
	\end{align*}
We may replace the upper limit of $\log^2n$ by $\infty$ and introduce a subpolynomial error. This leaves only the $j = 1$ term.  By strong spatial mixing we have \begin{equation*}
		\int_{\Lambda_-^{(1)}} (\rho_{\bl_{x_1}}(x) - \rho_{\blambda_{\ve_1}}(0)) \,dx  = o(n^{d-1})\,.
\end{equation*}
	Finally, we have \begin{equation*}
		\int_{\Lambda_+^{(1)}} (\rho_{\bl_{x_1}}(x) - \rho_{\blambda_{\ve_1}}(0)) \,dx = (2n)^{d-1} \int_{0}^{\infty}(\rho_{\bl_{\ve_1,t}}(0) - \rho_{\bl_{\ve_1}}(0))\,dt + o(n^{d-1})\,.
	\end{equation*}
	By Lemma \ref{lem:torus-alt}, the integral on the right-hand side is $0$, completing the proof.
	\end{proof}

	\section{Algorithms for the pressure and surface pressure}
	\label{secAlgorithms}
In this section we provide approximation algorithms for the pressure and surface pressure, proving Theorem~\ref{thm:pressure-algs}.	 We will combine Theorems \ref{thmSSM} and \ref{thmFastMixing} with Theorem~\ref{lem:pressure-dir} to demonstrate a randomized algorithm for the pressure in the $\lambda < e/ \Delta_\phi$ regime.  We note that strong spatial mixing along with \eqref{eq:pressure-identity} and \eqref{eq:surf-pressure-identity} imply uniform upper bounds on $p(\lambda)$ and the absolute value of the surface pressure for $\lambda \in [0,\lambda_0]$ for each $\lambda_0 < e/\Delta_\phi$, and thus we will prove approximations up to additive errors rather than multiplicative errors.

	First, we will use Theorem~\ref{thmFastMixing} to deduce a method of approximating densities.  Recall that by \eqref{eq:added-energy}, for any $\bl$ we may write the density at a point $v$ by $$\rho_{\bl}(v) = \bl(v) \E_{\bl} e^{-H_v(\bX)}\,.$$  Monte-Carlo approximation for the expectation using block dynamics will allow us to approximate $\rho$.
	\begin{lemma}\label{lem:MC}
		Let  for all $\bl \leq \lambda$.  Then there are constants $C,L > 0$ so that the following holds. Suppose  $\bl \leq \lambda$ is supported in a box of volume $n$ and  $v \in \supp(\bl)$.  For $\eps > 0$, let $\hat{\rho}$ be the random variable given by running block dynamics for density $\bl$ with radius $L$ for time $T = C n \log(n/\eps)$ a total of $N$ independent times to obtain point processes $\bX_1,\ldots,\bX_N$, and defining $$\hat{\rho} = \frac{\bl(v)}{N} \sum_{j = 1}^N e^{-H_v(\bX_j)}\,.$$
		Then $|\E \hat{\rho} - \rho_{\bl}(v)| \leq \eps$ and $\Var(\hat{\rho}) \leq \lambda^2/N$.
	\end{lemma}
	\begin{proof}
		Let $L$ be as guaranteed by Theorem \ref{thmFastMixing}, and choose $C$ large enough so that $\tau_{\mathrm{mix}}(\eps/\lambda) \leq T$.  Then since $\rho_{\bl}(v)$ is bounded by $\lambda$, we have that $|\E \hat{\rho}- \rho_{\bl}(v)| \leq \eps$, by definition of total variation.  Further, since $\hat{\rho} \leq \lambda$, we have $\Var(\hat{\rho}) \leq \lambda^2/ N$. 
	\end{proof}

	An algorithm for the pressure and proof of the first part of Theorem~\ref{thm:pressure-algs}  follows.  
	\begin{proof}[Proof of Theorem~\ref{thm:pressure-algs}, algorithm for the pressure]
		By Proposition \ref{lem:pressure-dir} we may write $p(\lambda) = \rho_{\bl_{\ve_1}}(0)$.  Set $s = C \log (1 / \eps)$ where $C$ will be large but fixed and define $S := [0,s] \times [-s,s]^{d-1}$ and $\bl$ by $\bl(x) = \lambda \{x \in S \}$.  For $C$ large enough, strong spatial mixing implies that $|\rho_{\bl_{\ve_1}}(0) - \rho_{\bl}(0)| \leq \eps/3$.    
		
		Applying Lemma \ref{lem:MC} shows that we may take $T = \Theta(\log^{d+1}(1/\eps) )$ and run block dynamics for time $T$ with radius $L$ a total of $m = \Theta(\eps^{-2})$ times in order to yield an approximation $\hat{\rho}$ with $|\E \hat{\rho} - \rho_{\bl}(0)| \leq \eps/3$ and $\Var(\hat{\rho}) \leq \eps^2/36$. By Chebyshev's inequality, we thus have $$\P(|\hat{\rho} - \E \hat{\rho}| \geq \eps/3) \leq \frac{1}{4}$$
		thus showing that this scheme provides the desired approximation.  The total runtime of this algorithm is $T\cdot m = \Theta(\eps^{-2}\log^{d+1}(1/\eps) )$.
	\end{proof}
	
	We now turn our attention to the surface pressure.  By Lemma \ref{lem:sp-box}, to find a randomized $\eps$-approximation algorithm to the surface pressure it is sufficient to find a randomized algorithm for $$\int_0^\infty (\rho_{\blambda_{\ve_1,\ve_2}}(t\ve_2) - \rho_{\blambda_{\ve_1}}(0))\,dt$$ that has error $\eps/(2(d-1))$ with probability at most $1/(8d)$.

	We first show that we can take a mesh to approximate integrals that appear in Lemma \ref{lem:sp-box} deterministically, from which it will be simpler to approximate the corresponding sum.

	\begin{lemma}\label{lem:int-to-sum}
		There are constants $C,c > 0$ so that for each $\eps > 0$ if we set $h = c \eps /\log^{d-1}(1/\eps), t_j = jh$ and $M = C \eps^{-1}\log^d(1/\eps)$ then we have $$\left|\int_0^\infty (\rho_{\bl_{\ve_1,\ve_2}}(t\ve_2) - \rho_{\bl_{\ve_1}}(0))\,dt - h\sum_{j = 1}^{M} (\rho_{\bl_{\ve_1,\ve_2}}(t_j\ve_2) - \rho_{\bl_{\ve_1}}(0)) \right|  \leq \eps\,.$$
	\end{lemma}
	\begin{proof}
		By strong spatial mixing there is a constant $C_1$ so that \begin{equation}
			\left|\int_0^\infty (\rho_{\bl_{\ve_1,\ve_2}}(t\ve_2) - \rho_{\bl_{\ve_1}}(0))\,dt - \int_0^{C_1 \log(1/\eps)} (\rho_{\bl_{\ve_1,\ve_2}}(t\ve_2) - \rho_{\bl_{\ve_1}}(0))\,dt \right| \leq \frac{\eps}{2}\,.
		\end{equation}
		
		Turning our attention to the truncated integral, we will prove that the integrand is Lipschitz, which will allow approximation by a sum easily.  
		\begin{claim}\label{cl:lipschitz}
			For each $s,t \geq 0$ we have $$|\rho_{\bl_{\ve_1,\ve_2}}(t\ve_2) - \rho_{\bl_{\ve_1,\ve_2}}(s\ve_2)| =O(\eps^2 + \log^{d-1}(1/\eps)|s-t|)\,.$$
		\end{claim}	
		\begin{proof}
			Choose $C_2$ large enough so that if we set $\bl^{(t)}(x) := \bl_{\ve_1,\ve_2}(x) \one\{\|x - t\ve_2\| \leq C_2 \log(1/\eps)  \}$ when we have $\rho_{\bl_{\ve_1,\ve_2}}(t\ve_2) = \rho_{\bl^{(t)}}(t\ve_2) + O(\eps^2)$. Define $\bl^{(s)}$ similarly.  Applying Corollary \ref{cor:compare} completes the proof.
		\end{proof}
		
		Applying the claim and choosing $c$ small enough completes the proof.
	\end{proof}
	
	We now prove the second part of Theorem~\ref{thm:pressure-algs}, giving an algorithm for the surface pressure.
	
	\begin{proof}[Proof of Theorem~\ref{thm:pressure-algs}, algorithm for surface pressure]
		It is sufficient to find a randomized algorithm that with probability at least $1 - 1/(8d)$ is within $\eps/(2(d-1))$ of $\int_0^\infty (\rho_{\bl_{\ve_1,\ve_2}}(t\ve_2) - \rho_{\bl_{\ve_1}}(0))\,dt$.  With this in mind, define $f(t) = \rho_{\bl_{\ve_1,\ve_2}}(t\ve_2)$.  As in the proof of Claim \ref{cl:lipschitz}, for each density that we wish to approximate, zero out the activity function outside of the ball of radius $C_1 \log(1/\eps)$ centered at $t_j \ve_2$.  Apply Lemma \ref{lem:MC} and let $\hat{f}(t)$ denote the random variable given by running block dynamics in this ball for time $T = C_2\log^{d+1}(1/\eps)$ a total of $N$ times to approximate $f(t)$.  Then note that $$\mathrm{Var}\left(h \sum_{j = 1}^M \hat{f}(t_j)\right) \leq h^2 M \max_j \mathrm{Var}( \hat{f}(t_j)) =O\left(\eps \log^{2-d}(1/\eps) N^{-1}  \right)\,. $$
	    For each $c > 0$ we may take $C_2$ large enough so that $|f(t_j) - \E \hat{f}(t_j)| \leq c \eps / \log(1/\eps)$; in particular, we may choose $C_2$ large enough so that $$\left|\E h \sum_{j = 1}^M \hat{f}(t_j)  - h \sum_{j = 1}^M {f}(t_j)  \right| \leq \frac{\eps}{4(d-1)}\,.$$
	    
	    Further, if we take $N = C_3 \eps^{-1} \log^{2- d}(1/\eps)$, then we may assure $$\P\left(\left| h \sum_{j = 1}^M \hat{f}(t_j)  - h \sum_{j = 1}^M \E\hat{f}(t_j)  \right|  \geq \eps/2 \right)  \leq \frac{1}{8d}\,.$$
	    
	    This shows that the given scheme provides the desired random approximation algorithm.  It has total running time 
	    \[ (d-1)\cdot T\cdot N \cdot M = \Theta( \eps^{-2} \log^{d+3}(1/\eps)) \,. \qedhere \]
	\end{proof}

This completes the proof of Theorem~\ref{thm:pressure-algs}.

	\section*{Acknowledgments}
	MM supported in part by NSF grant DMS-2137623.  WP supported in part by NSF grant DMS-1847451.  We thank Marcus Pappik for helpful comments on the paper.
	
\newcommand{\etalchar}[1]{$^{#1}$}

\appendix

\section{Strong spatial mixing implies fast mixing}	
\label{secSSMtoBlock}

Theorem~\ref{thmFastMixing}, our result on fast mixing of the block dynamics  follows directly from two ingredients: strong spatial mixing established in Theorem~\ref{thmSSM} above, and the following theorem establishing that fast mixing of block dynamics is a consequence of strong spatial mixing.  

\begin{theorem}
\label{thmSSMtoMix}
If a Gibbs point process with a finite-range, repulsive potential $\phi$ exhibits strong spatial mixing on $\R^d$ for activities bounded by $\lam$, then there exists $L_0 >0$ so that for all $L \ge L_0$ the block dynamics with update radius $L$ for $\mu_{\Lam_n, \bl}$ has mixing time $\tau_{\mathrm{mix}}(\eps) = O(N \log (N/\eps))$ when $\bl $ is bounded by $\lam$, where $N = |\Lam_n|$. 
\end{theorem}	


The proof of this theorem closely follows that of~\cite[Theorem 8]{helmuth2020correlation} which in turn is adapted from~\cite[Theorem 2.5]{dyer2004mixing}.  The main change is the choice of a different metric that can handle the fact that the number of points appearing in a bounded region is unbounded if the potential does not have a hard core.

We  use the path coupling method of Bubley and Dyer~\cite{bubley1997path}.   Consider a discrete time Markov chain on a state space $\Omega$.  We place a graph structure  $G_{\Omega}$ on $\Omega$ by declaring some pairs of configurations to be adjacent and define a \textit{pre-metric} $\hat D$ on pairs of adjacent configurations.  The pre-metric should be symmetric and bounded below by $1$. We then extend this pre-metric to a \textit{path metric} $D$ on all pairs of configurations by taking the shortest path distance (with respect to $\hat D$) on the graph $G_{\Omega}$.  We require that $D(X,Y) = \hat D(X,Y)$  for adjacent configurations $X,Y$ (in our case this will follow directly from our definition of $\hat D$).   The \textit{diameter} of $G_{\Omega}$ with respect to $D$ is $\text{diam}(G_\Omega) = \sup_{X,Y \in \Omega} D(X,Y)$.   We use the following version of the path coupling technique.

\begin{lemma}[{\cite[Corollary 14.7]{levin2017markov}}]
  \label{lemPathCouple}
  Consider a discrete time Markov chain with finite or infinite state space $\Omega$.  Define a graph structure $G_{\Omega}$ on $\Omega$ along with a corresponding pre-metric $\hat D$ and path metric $D$.  Suppose that $\mathrm{diam}(G_\Omega)  < \infty$. 

  Suppose that for each pair $\{X_0, Y_0\}$ of adjacent configurations 
  the following holds: there exists a coupling $(X_1, Y_1)$ of the distributions of the one-step distributions of the Markov chain with initial configuration $X_0$ and $Y_0$ respectively such that
  \begin{equation*}
    \E [D(X_1, Y_1)]\leq  D(X_0, Y_0)e^{-\xi} =  \hat D(X_0, Y_0)e^{-\xi} \,.
  \end{equation*}
   Then
  \begin{equation*}
    t_{\mathrm{mix}}(\varepsilon) \leq \left \lceil \frac{ \log(\mathrm{diam}(G_\Omega)) + \log(1/\varepsilon )}{\xi} \right
    \rceil.
  \end{equation*}
\end{lemma}

We now can prove Theorem~\ref{thmSSMtoMix}.

\begin{proof}[Proof of Theorem~\ref{thmSSMtoMix}]
For a finite-range repulsive potential $\phi$ with range at most $r$, assume that strong spatial mixing holds with constants $\alpha$ and $\beta$.  Set the update radius to be  $L = Kr$ where $K$ will be chosen sufficiently large later.  Let $V= |B_r|$ be the volume of the ball of radius $r$.   Let $N = |\Lam_n|$ be the volume of the ball of radius $n$. 

Now consider the Gibbs point process on $\Lam_n$ with activity function $\bl$.   Let $\Omega$ be the set of finite subsets of $\Lam_n$.  Define a graph structure on $\Omega$ by declaring $X, Y \in \Omega$ to be adjacent if the symmetric difference $X \triangle Y$ is contained in a ball of radius $r$ around some point in $\Lam_n$.  Then define the pre-metric on adjacent states $\hat D (X,Y) \equiv 1$.  This extends naturally to  a path metric $D$:  for $X, Y \in \Omega$,  $D(X,Y)$ is the minimum number of balls $B_r(u_1),\ldots,B_r(u_k)$ of radius $r$ so that $X$ and $Y$ agree away from $\bigcup_j B_r(u_j)$. Note that  since $\Lam_n$ can be covered by at most $C_d N$ balls of radius $r$ (where the constant $C_d$ depends on $d$ and $r$), we have $\diam(G_\Omega) \leq  C_d N$.

	Let $X_t$ and $Y_t$ be two radius-$L$ block dynamic chains for $\mu_{\bl}$ with $X_0 \setminus B_r(u) = Y_0 \setminus B_r(u)$, i.e. the two configurations are adjacent in $G_{\Omega}$ and their disagreements are contained in the ball of radius $r$ centered at $u$.  We will couple the two chains as follows: we choose the same update ball in each chain; if the boundary conditions agree then we make the same update.  If they disagree, we will choose a coupling  described below.  
	
	Define $\Delta  =  D(X_1,Y_1) - D(X_0,Y_0) =D(X_1,Y_1) - 1 $ and let $x$ be the random center of the update ball.  If $B_r(u) \subset B_L(x)$ i.e.\ $u \in B_{L-r}(x)$, then the boundary conditions in $X_0$ and $Y_0$ agree and so $X_1 = Y_1$; on this event, we have $\Delta = -1$.  Further, this event occurs with probability $$\Pr[B_r(u) \subset B_L(x)] = \frac{|B_{L-r}(u) \cap \Lambda_n|}{|\Lambda_n|} \geq \frac{(K-1)^d V}{N} \,. $$
	If $B_r(u) \cap B_{L + r}(x) = \emptyset$ i.e.\ $u \notin B_{L+2r}(x)$, then the boundary conditions of the update ball again agree, and so $X_1$ and $Y_1$ agree away from $B_r(u)$, and so $\Delta = 0$. 
	
	The remaining case is when $u \in B_{L+2r}(x) \setminus B_{L-r}(x)$.  In this case, the boundary conditions of the update ball may differ and so we may have $\Delta > 0$.  We bound the probability that this placement of $x$ occurs: $$\Pr[u \in B_{L + 2r}(x) \setminus B_{L -r}(x)] \leq \frac{(K +2 )^d V- (K-1)^dV }{N} \leq c_d \frac{K^{d-1}V}{N}\,.$$
We now bound the expected increase in $D$ in this case under a coupling provided below.
	
	Let $x \in \Lambda_n$ with $u \in B_{L+2r}(x) \setminus B_{L-r}(x)$.  Write $\tau_X$ for the boundary condition induced by $X_0$ and $\tau_Y$ for the condition induced by $Y_0$.  Note that by assumption $Y_0 \triangle X_0 \subset B_r(u)$.  For a parameter $t$ to be chosen later, let $A = B_L(x) \cap B_t(u)$ and set $\Abar = B_{L}(x) \setminus A$.
	
	Write $D(X_1,Y_1) \leq D(X_1 \cap \Abar,Y_1 \cap \Abar) +  D(X_1 \cap A,Y_1 \cap A)$.  We will first update both $X_1$ and $Y_1$ in $\Abar$ and then update both configurations within $A$ independently.  Note that deterministically we have $ D(X_1 \cap A,Y_1 \cap A) \leq C_d |A| \leq C_d t^d\,.$
	
	We now describe the coupling to update within $\Abar$.  The total variation distance between $\mu_{B_L(x)}^{\tau_X}$ and $\mu_{B_L(x)}^{\tau_Y}$ restricted to $\Abar$ may be bounded using the strong spacial mixing assumption, i.e. \begin{align*}
	\|\mu_{B_L(x)}^{\tau_X} -\mu_{B_L(x)}^{\tau_X}  \|_{\Abar} &\leq \alpha | \Abar| \exp\left( -  \beta \cdot \dist(\tau_X \triangle \tau_Y, \Abar)\right) \\
	&\leq \alpha |\Abar| \exp\left(-\beta \cdot \dist(B_r(u),\Abar \right) \\ 
	&\leq \alpha  K^d V \exp\left(-\beta(t-r) \right)\,.
	\end{align*}
	
	Thus there exists a coupling of $X_1, Y_1$ so that they disagree within $\Abar$ with probability at most $\alpha K^d V \exp(-\beta(t- r))$.  On the event that they agree, $\Delta = 0$, and so we need only handle the case in which they fail to couple.  An upper bound on the increase in $D$ within $\Abar$ is $C_d K^d V$, and so $$ \E[ \Delta \,|\, u \in B_{L+2r} \setminus B_{L-r}(x)] \leq  C_d \left( t^d + \alpha K^{2d} e^{-\beta(t-r)} \right)\,.$$
	Thus, we may bound the expectation of $\Delta$: \begin{align*}\E[\Delta] \leq -\frac{(K-1)^d V}{ N} + c_d C_d\frac{K^{d-1} V}{N}\left(  t^d + \alpha K^{d} V e^{-\beta(t-r)} \right)\,.
	\end{align*}
	Pick $t = \gamma K^{1/d}$ where we will choose $\gamma > 0$ sufficiently small momentarily; then \begin{align*}
	\E[\Delta] &\leq -\frac{V}{N}\left( (K-1)^d - c_dC_d K^{d-1} (\gamma^d K + \alpha K^{d} V e^{- \beta \gamma K^{1/d} +\beta r  } )\right) \\
	&\leq -\frac{K^d V}{ N} \left( (1 - 1/K)^d - c_d C_d \gamma^d - \alpha  K^{2d} V e^{- \beta \gamma K^{1/d} + \beta r} \right) \\
	&\leq -\frac{K^d V}{ N} \left(\frac{1}{2} - \alpha K^{2d} V e^{- \beta \gamma K^{1/d} + \beta r} \right) \end{align*}
	where the last inequality holds by assuming $K \geq 2$ and choosing $\gamma$ small enough as a function of $d$.  Taking $K$ sufficiently large depending on $\alpha,\beta,d$ and $r$ yields $$\E[\Delta] \leq -\frac{K^d V}{4 N}\,.$$
Applying Lemma~\ref{lemPathCouple} with $\xi = \frac{K^d V}{4 N}$ and $\text{diam}(G_{\Omega}) \le C_d N$	   then yields Theorem~\ref{thmSSMtoMix}.
\end{proof}

\section{Approximate counting}
\label{secApproxCount}

Here we prove Corollary~\ref{thmApproxAlg} following a standard reduction of approximate counting to approximate sampling.  In this section the activity function $\bl$ will be fixed, so we drop it from the notation, writing $Z_{\Lam}$ for $Z_{\Lam}(\bl)$ and $\mu_{\Lam}$ for $\mu_{\Lam, \bl}$. 

Recall that $N = |\Lam_n|$. 
We partition $\Lam_n$ into $N$ boxes of volume $1$, 
$$\Lam_n = S_1 \uplus S_2  \uplus \cdots \uplus S_N \, . $$
For $k = 0, \dots , N$, let $\Lam^{(k)} = \Lam_n \setminus \left( \bigcup_{j=1}^k S_j  \right) $.  In particular, $\Lam^{(0)} = \Lam_n$ and $\Lam^{(N)}= \emptyset$.  

We aim to estimate $Z_{\Lam_n} $ which can be written as the inverse of the probability of seeing no points in the random point set $\mathbf X$:
\[ Z_{\Lam_n} = \frac{1}{\mu_{\Lam_n} (  \mathbf X = \emptyset  )  }  \,,\]
so it will suffice to approximate  $\mu_{\Lam_n} (  \mathbf X = \emptyset  ) $.   

By the spatial Markov property of a Markov random field, the law of $\mu_{\Lam_n}$ conditioned on the event  $\{ \mathbf X \cap \bigcup_{j=1}^k S_k = \emptyset \}$ is $\mu_{\Lam^{(k)}}$.   We can then write
\begin{align*}
\mu_{\Lam_n} (  \mathbf X = \emptyset  ) &= \prod_{j=0}^{N-1}   \mu_{\Lam^{(k)}}  \left( \mathbf X \cap S_{j+1} = \emptyset  \right)  \,.
\end{align*}

We first show how to approximate $\mu_{\Lam_n} (  \mathbf X = \emptyset  )$ given access to an exact sampler from the Gibbs point process; we then will extend this to work with the approximate sampler given by Theorem~\ref{thmFastMixing}.  
\begin{lemma}
\label{lemExactCounter}
There is a randomized algorithm that with probability at least $7/8$ produces an $\eps$-relative approximation to $\mu_{\Lam_n} (  \mathbf X = \emptyset  )$ given access to exact samples from $\mu_{\Lam^{(k)} }$ for $k=0, \dots, N$.  The algorithm uses $O(N^2/\eps^2)$ such samples and runs in time $O(N^2/\eps^2)$. 
\end{lemma}
\begin{proof}
Let $T = C N/\eps^2 $.  For $j=0, \dots , N-1$, let $\overline Y_j = ( Y_{1j} + \dots + Y_{Tj}  )/T$ where $Y_{ij}$ is the indicator  that $\mathbf X \cap S_{j+1} = \emptyset$ in a sample from $\mu_{\Lam^{(k)}}$, with all samples taken independently.  Then by construction $\E \overline Y_j =  \mu_{\Lam^{(k)}}  \left( \mathbf X \cap S_{j+1} = \emptyset  \right)$, and $\var(\overline Y_j ) \le \E \overline Y_j /T$.   Using Poisson domination we can  lower bound  $\mu_{\Lam^{(k)}}  \left( \mathbf X \cap S_{j+1} = \emptyset  \right)$ by $e^{-\lam}$ and so obtain
\begin{align*}
\frac{ \var( \overline Y_j)  }{ (\E  \overline Y_j)^2 } \le  \frac{e^{\lam}}{    T} \,.
\end{align*}

 Now let
\begin{align*}
W =\prod_{j=0}^{N-1}  \frac{ \overline Y_j  }{     \mu_{\Lam^{(k)}}  \left( \mathbf X \cap S_{j+1} = \emptyset  \right)  } \,.
\end{align*}
We have $\E W =1$, and 
\begin{align*}
\var(W) & = \prod_{j=0}^{N-1}   \frac{  \E [ \overline Y_j^2] }{  \mu_{\Lam^{(k)}}  \left( \mathbf X \cap S_{j+1} = \emptyset  \right) ^2  }   -1  =  \prod_{j=0}^{N-1}  \left( 1 +   \frac{  \var (Y_j) }{  \mu_{\Lam^{(k)}}  \left( \mathbf X \cap S_{j+1} = \emptyset  \right) ^2  }  \right)  -1  \\
&\le \left( 1+   \frac{e^{\lam}}{    T}  \right)^N -1 \le \exp ( N e^{\lam} /T) -1 \le 2 e^{\lam}\eps^2  /C \,. 
\end{align*}
Now by Chebyshev's inequality,
\begin{align*}
\Pr [ e^{-\eps} \le W \le e^{\eps} ] &\ge 1-  \Pr[ | W -1| \ge  \eps/2 ]  \ge 1- \frac{4 \var(W)  }{ \eps^2    } \ge 1-  \frac{ 8 e^{\lam}\eps^2   }{  C\eps^2    }  \, ,
\end{align*}
which for $C$ large enough is at least $7/8$, and so outputting $\prod_{j=0}^{N-1} \overline Y_j$ gives the desired approximation.   The number of samples used is $N T$ and the running time is $O(NT ) = O( N^2/\eps^2)$. 
\end{proof}

With this we can prove Corollary~\ref{thmApproxAlg}. 
\begin{proof}[Proof of Corollary~\ref{thmApproxAlg}]  Set $T = C N / \eps^2$ as in the proof of Lemma \ref{lemExactCounter}.
Theorem~\ref{thmFastMixing} gives us a sampler accurate to within total variation distance $(8NT)^{-1}$ with each sample taking a running time of $O( N \log (n^2 T)) = O( N \log (N/\eps))$.   Using the properties of total variation distance, there is a coupling of $NT$ independent samples from this sampler and the $NT$ independent exact samples  used in Lemma~\ref{lemExactCounter} so that the sequences agree with probability at least $7/8$.   In particular, running the algorithm of Lemma~\ref{lemExactCounter} with this approximate sampler yields an $\eps$-relative approximation to $Z_{\Lam_n}$ with probability at least $3/4$.  The running time is $O( N^3 \eps^{-2}  \log (N/\eps))$.   
\end{proof}

\end{document}